\documentclass[10pt]{article}

\usepackage{amssymb,amsthm,amsmath}
\usepackage{mathtools}
\usepackage[thmtools-compat]{keytheorems}

\usepackage[margin=1in]{geometry}

\usepackage{hyperref}
\usepackage[english]{babel}
\usepackage[style=numeric, giveninits=true, maxbibnames=99]{biblatex}
\usepackage{csquotes}

\usepackage{lmodern}
\usepackage{enumitem}

\usepackage{tikz}
    \usetikzlibrary{shapes.geometric, arrows, calc, positioning}
    \usepackage{varwidth}
    \tikzstyle{startstop} = [rectangle, rounded corners, draw=black, text centered, inner sep=0.5em]
    \tikzstyle{process} = [rectangle, draw=black, text centered, inner sep=0.5em]
    \tikzstyle{longprocess} = [rectangle, draw=black, text centered, text width=4cm, inner sep=0.5em]
    \tikzstyle{io} = [trapezium, 
    trapezium stretches=true, trapezium left angle=70, trapezium right angle=110, draw=black, text centered, inner sep=0.5em]
    \tikzstyle{longio} = [trapezium, 
    trapezium stretches=true, trapezium left angle=70, trapezium right angle=110, draw=black, text centered, text width=3cm, inner sep=0.5em]
    \tikzstyle{decision} = [diamond, aspect=2, draw=black, text centered, inner sep=0.2em]
    \tikzstyle{longdecision} = [diamond, aspect=2, draw=black, text centered, text width=3cm, inner sep=0.2em]
    \tikzstyle{arrow} = [thick,->,>=stealth]
    \usepackage{pgfplots}
    \pgfplotsset{compat=1.18}
    
\usepackage{accsupp} % this allows alttext for images

\usepackage{todonotes}

\newcommand\RR{\mathbb{R}}
\newcommand\ZZ{\mathbb{Z}}
\newcommand\NN{\mathbb{N}}
\newcommand{\cA}{{\cal A}}
\newcommand{\cF}{{\cal F}}
\newcommand{\cG}{{\cal G}}

\newcommand{\Prob}[1]{{\mathbf P}\!\left[#1\right]}
\newcommand{\EV}[1]{{\mathbf E}\!\left[#1\right]}

\DeclareMathOperator{\diam}{diam}
\DeclareMathOperator{\diag}{diag}
\DeclareMathOperator{\supp}{supp}
\DeclareMathOperator{\tr}{tr}
\DeclarePairedDelimiter{\abs}{|}{|}
\DeclarePairedDelimiter{\norm}{\|}{\|}
\DeclarePairedDelimiter{\angleb}{\langle}{\rangle}

\makeatletter
\newcommand*\wt[2][0.1ex]{%
        \begingroup
        \mathchoice{\wt@helper{#1}{#2}{\displaystyle}{\textfont}}
                   {\wt@helper{#1}{#2}{\textstyle}{\textfont}}
                   {\wt@helper{#1}{#2}{\scriptstyle}{\scriptfont}}
                   {\wt@helper{#1}{#2}{\scriptscriptstyle}{\scriptscriptfont}}%
        \endgroup
        #2%
}
\newcommand*\wt@helper[4]{%
        \def\currentfont{\the#41}%
        \def\currentskewchar{\char\the\skewchar\currentfont}%
        \setbox\tw@\hbox{\currentfont#2\currentskewchar}%
        \dimen@ii\wd\tw@
        \setbox\tw@\hbox{\currentfont#2{}\currentskewchar}%
        \advance\dimen@ii-\wd\tw@
        \rlap{\raisebox{-#1}{$\m@th#3\kern\dimen@ii\widetilde{\phantom{#2}}$}}%
}
\makeatother
\renewcommand{\tilde}[1]{\smash{\wt{#1}}}

\newcommand*{\1}{\text{\usefont{U}{bbold}{m}{n}1}}
\newcommand{\pn}[1]{\textsf{#1}}

\title{On the diameter and distribution of Sidon sets}
\date{2026}
\author{
Daniel~Carter\thanks{Princeton University, \href{mailto:dc65@princeton.edu}{\tt dc65@princeton.edu}.},\quad
Kevin~O'Bryant\thanks{City University of New York, College of Staten Island and The Graduate Center, \href{mailto:kevin.obryant@csi.cuny.edu}{\tt kevin.obryant@csi.cuny.edu}.},\quad
Terence~Tao\thanks{University of California, Los Angeles, \href{mailto:tao@math.ucla.edu}{\tt tao@math.ucla.edu}.}}

\declaretheorem[numberwithin=section]{theorem}
\declaretheorem[numbered=no, name=Theorem]{theorem*}

\declaretheorem[numbered=no, name=Claim]{claim*}
\declaretheorem{claim}

\declaretheorem[numberlike=theorem]{lemma,proposition,problem,corollary,conjecture,question}
\declaretheorem[numberlike=theorem, style=definition, qed=$\lozenge$]{definition,example,remark}

\begin{document}
\maketitle
\begin{abstract}
A Sidon set is a set $\cA$ of integers with the property that if $0\neq a-b=c-d$ and $a,b,c,d\in\cA$, then $(a,b)=(c,d)$. 
Our main result is an improved and explicit bound on the diameter of a Sidon set: if $\cA$ is a Sidon set with $k\ge 1$ elements, then
  \[\diam(\cA) \ge k^2 \left(\frac{1}{1 + 0.945/\sqrt{k}}-\frac{0.945}{\sqrt k}\right) -1;\]
in particular, $\diam(\cA)\ge k^2-1.89 k^{3/2}$. This implies that if $\cA\subseteq[1,N]$ is a Sidon set, then $\abs{\cA} \le N^{1/2}+0.945 \, N^{1/4}+0.447$. Our technique also gives an improved bound for the case of large $k$, namely $\diam(\cA) \ge k^2-1.88607 k^{3/2}$ for $k\ge 483$. A $\gamma$-Golomb ruler is a set $\cA$ of integers with the property that there are not $\gamma+1$ distinct pairs $(a_i,b_i)\in\cA^2$ with $0\neq a_1-b_1=\cdots=a_{\gamma+1}-b_{\gamma+1}$. Our diameter bounds extend in a natural way to the case of Golomb rulers. Our techniques can also be adapted to address a number of questions on the distribution of Golomb rulers: we provide upper bounds on the two quantities $\abs{a_\ell-a_1 -\frac{\ell-1}{k-1}(a_k-a_1)}$ and $\max_{1\le i<k} (a_{i+1}-a_i)$, where $\cA=\{a_1<\dots<a_k\}$ is a $\gamma$-Golomb ruler with $\diam(\cA)\le (k^2-bk^{3/2})/\gamma$ (where $b\in \RR$ is a free parameter).
\end{abstract}

\section{Introduction}\label{sec:introduction}
\subsection{Sidon sets}
A Sidon set is a set $\cA$ of integers with the property that if $a,b,c,d\in\cA$ and $a+b=c+d$, then $\{a,b\}=\{c,d\}$. An equivalent condition is that if $a,b,c,d\in\cA$ and $a-b=c-d$, then either $a-b=c-d=0$ or $(a,b)=(c,d)$. Sidon sets (also known as Golomb rulers, $1$-Golomb rulers, $B_2$-sets, and Babcock sets) have a vast literature: the annotated bibliography~\cite{2004.Obryant} has almost 200 works, and a new version will appear shortly with over 400. 

The fundamental problem is to bound the diameter in terms of the cardinality of finite Sidon sets. Our headline result is the following explicit bound:
\begin{corollary}[The minimum diameter of a Sidon set]
\label{cor:Sidon}
  If $\cA$ is a Sidon set with $k\ge 2$ elements, then
    \[\diam(\cA) \ge k^2\left(\frac{1}{1+0.945/\sqrt k}-\frac{0.945}{\sqrt k}\right)-1 \ge k^2-1.89\,k^{3/2}.\]
  If $\cA\subseteq[1,N]$ is a Sidon set, then
	\[|\cA| \le N^{1/2}+0.945\,N^{1/4}+0.447.\]
\end{corollary}
In the case where $k$ is sufficiently large, we improve our numbers slightly:
\begin{corollary}\label{cor:asymptotic}
If $\cA\subseteq[1,N]$ is a Sidon set with $k\ge 483$ elements, then $\diam(\cA)\ge k^2-1.88607\, k^{3/2}$ and $k\le N^{1/2}+0.94304\,N^{1/4}+0.44466$.
\end{corollary}
\noindent The previous asymptotically best lower bound on the diameter was the non-explicit 
\[ \diam(\cA)\ge k^2 - 1.98183 \, k^{3/2}-O(k) \]
found in~\cite{2025.Carter&Hunter&Obryant}. While this paper was in preparation, Hou and Zhao~\cite{2026.Hou&Zhao} announced an improvement to
\[ |\cA| \le N^{1/2} + 0.9435\,N^{1/4}+O(1). \]

Define $b_k$ to be the largest number such that there is a Sidon set $\cA$ with $k$ elements and $\diam(\cA)=k^2-b_k k^{3/2}$, and let
\[ \overline{b_\infty}=\limsup_{k\to\infty} b_k. \]
An inequality in the work of Erd\H{o}s and Tur\'an~\cite{1941.Erdos&Turan} in 1941 establishes that $b_k\le 2$ for all $k\ge 1$, and Singer~\cite{1938.Singer} proved in 1938 that $\overline{b_\infty}\ge 0$. The aforementioned asymptotic results imply $\overline{b_\infty}\le 1.98183$ and $\overline{b_\infty}\le 1.887$, respectively, and our work improves this slightly to $\overline{b_\infty}\le 1.88607$, while also providing explicit bounds on $b_k$ for finite $k$.

The value of $b_k$ has been computed exactly for $1\le k \le 28$, with $b_{28}=\frac{199}{56 \sqrt{7}}\approx 1.343$ alone taking tens of thousands of machines more than $8$ years to complete~\cite{distributed.net}. Figure~\ref{fig:b_k values and bounds} shows the exact values, the upper bounds proved in this work, and the lower bounds computed by Shearer~\cite{shearer_ogr_table}.

We note that Shearer's technique (recently extended to $k\le 40\,200$ by Dogon and Rokicki~\cite{2016.Dogon&Rokicki}) happened to find the actual $b_k$ value for $17\le k \le 28$, suggesting that the lower bounds for $k\ge 29$ may be the truth. Eberhard and Manners~\cite{2023.Eberhard&Manners} give a plausible argument that the lower bound is the truth. On the other hand, Shearer's search relies on having a prime power near $k$, and as prime powers become scarcer perhaps other as-yet-unknown opportunities for dense Sidon sets arise. 

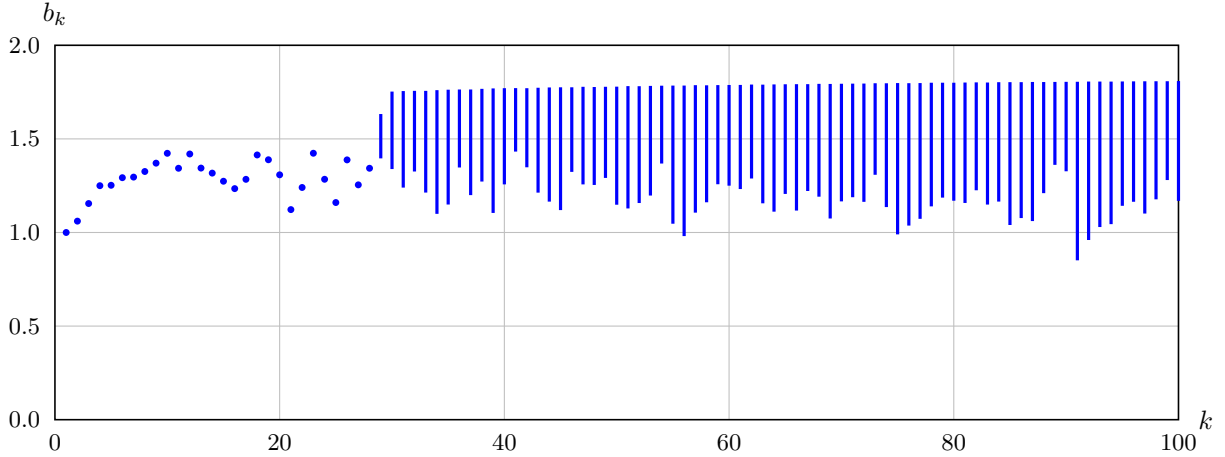
\begin{figure}
    \centering
    \begin{tikzpicture}[x=0.009\textwidth, y=0.15\textwidth]  
      % Gridlines (vertical)
      \draw[gray!50, line width=0.4pt] (20,0) -- (20,2.0);
      \draw[gray!50, line width=0.4pt] (40,0) -- (40,2.0);
      \draw[gray!50, line width=0.4pt] (60,0) -- (60,2.0);
      \draw[gray!50, line width=0.4pt] (80,0) -- (80,2.0);
      \draw[gray!50, line width=0.4pt] (100,0) -- (100,2.0);
    
      % Gridlines (horizontal)
      \draw[gray!50, line width=0.4pt] (0,0.5) -- (100,0.5);
      \draw[gray!50, line width=0.4pt] (0,1.0) -- (100,1.0);
      \draw[gray!50, line width=0.4pt] (0,1.5) -- (100,1.5);
      \draw[gray!50, line width=0.4pt] (0,2.0) -- (100,2.0);
    
      % Frame
      \draw[black, line width=0.6pt] (0,0) rectangle (100,2.0);
    
      % X tick labels
      \node[anchor=north, yshift=-2pt] at (0,0) {\small 0};
      \node[anchor=north, yshift=-2pt] at (20,0) {\small 20};
      \node[anchor=north, yshift=-2pt] at (40,0) {\small 40};
      \node[anchor=north, yshift=-2pt] at (60,0) {\small 60};
      \node[anchor=north, yshift=-2pt] at (80,0) {\small 80};
      \node[anchor=north, yshift=-2pt] at (100,0) {\small 100};
    
      % Y tick labels
      \node[anchor=east, xshift=-2pt] at (0,0.0) {\small 0.0};
      \node[anchor=east, xshift=-2pt] at (0,0.5) {\small 0.5};
      \node[anchor=east, xshift=-2pt] at (0,1.0) {\small 1.0};
      \node[anchor=east, xshift=-2pt] at (0,1.5) {\small 1.5};
      \node[anchor=east, xshift=-2pt] at (0,2.0) {\small 2.0};
    
      % Axis labels (placed at the ends of the axes, just past the frame)
      \node[anchor=west, xshift=4pt] at (100,0) {$k$};
      \node[anchor=south, yshift=4pt] at (0,2.0) {$b_k$};
    
      % Data: bars for 1<=k<=28 collapse to points (L1=L2), drawn as fixed-size
      % circular nodes so they stay round and legible no matter how \figwidthfactor
      % is set; bars (line segments) for k>=29.
      \node[fill=blue, circle, inner sep=0pt, minimum size=2.6pt] at (1,1.000000) {};
      \node[fill=blue, circle, inner sep=0pt, minimum size=2.6pt] at (2,1.060660) {};
      \node[fill=blue, circle, inner sep=0pt, minimum size=2.6pt] at (3,1.154701) {};
      \node[fill=blue, circle, inner sep=0pt, minimum size=2.6pt] at (4,1.250000) {};
      \node[fill=blue, circle, inner sep=0pt, minimum size=2.6pt] at (5,1.252198) {};
      \node[fill=blue, circle, inner sep=0pt, minimum size=2.6pt] at (6,1.292786) {};
      \node[fill=blue, circle, inner sep=0pt, minimum size=2.6pt] at (7,1.295878) {};
      \node[fill=blue, circle, inner sep=0pt, minimum size=2.6pt] at (8,1.325825) {};
      \node[fill=blue, circle, inner sep=0pt, minimum size=2.6pt] at (9,1.370370) {};
      \node[fill=blue, circle, inner sep=0pt, minimum size=2.6pt] at (10,1.423025) {};
      \node[fill=blue, circle, inner sep=0pt, minimum size=2.6pt] at (11,1.343096) {};
      \node[fill=blue, circle, inner sep=0pt, minimum size=2.6pt] at (12,1.419319) {};
      \node[fill=blue, circle, inner sep=0pt, minimum size=2.6pt] at (13,1.344081) {};
      \node[fill=blue, circle, inner sep=0pt, minimum size=2.6pt] at (14,1.317216) {};
      \node[fill=blue, circle, inner sep=0pt, minimum size=2.6pt] at (15,1.273781) {};
      \node[fill=blue, circle, inner sep=0pt, minimum size=2.6pt] at (16,1.234375) {};
      \node[fill=blue, circle, inner sep=0pt, minimum size=2.6pt] at (17,1.284012) {};
      \node[fill=blue, circle, inner sep=0pt, minimum size=2.6pt] at (18,1.414214) {};
      \node[fill=blue, circle, inner sep=0pt, minimum size=2.6pt] at (19,1.388569) {};
      \node[fill=blue, circle, inner sep=0pt, minimum size=2.6pt] at (20,1.308100) {};
      \node[fill=blue, circle, inner sep=0pt, minimum size=2.6pt] at (21,1.122263) {};
      \node[fill=blue, circle, inner sep=0pt, minimum size=2.6pt] at (22,1.240441) {};
      \node[fill=blue, circle, inner sep=0pt, minimum size=2.6pt] at (23,1.423338) {};
      \node[fill=blue, circle, inner sep=0pt, minimum size=2.6pt] at (24,1.284281) {};
      \node[fill=blue, circle, inner sep=0pt, minimum size=2.6pt] at (25,1.160000) {};
      \node[fill=blue, circle, inner sep=0pt, minimum size=2.6pt] at (26,1.387899) {};
      \node[fill=blue, circle, inner sep=0pt, minimum size=2.6pt] at (27,1.254489) {};
      \node[fill=blue, circle, inner sep=0pt, minimum size=2.6pt] at (28,1.343124) {};
      \draw[blue, line width=1.2pt] (29,1.395917) -- (29,1.632838);
      \draw[blue, line width=1.2pt] (30,1.338877) -- (30,1.752712);
      \draw[blue, line width=1.2pt] (31,1.239856) -- (31,1.755497);
      \draw[blue, line width=1.2pt] (32,1.325825) -- (32,1.756718);
      \draw[blue, line width=1.2pt] (33,1.213269) -- (33,1.756602);
      \draw[blue, line width=1.2pt] (34,1.099609) -- (34,1.760383);
      \draw[blue, line width=1.2pt] (35,1.149410) -- (35,1.762750);
      \draw[blue, line width=1.2pt] (36,1.347222) -- (36,1.763889);
      \draw[blue, line width=1.2pt] (37,1.199668) -- (37,1.763957);
      \draw[blue, line width=1.2pt] (38,1.272157) -- (38,1.767360);
      \draw[blue, line width=1.2pt] (39,1.104474) -- (39,1.769621);
      \draw[blue, line width=1.2pt] (40,1.257005) -- (40,1.770875);
      \draw[blue, line width=1.2pt] (41,1.432228) -- (41,1.771239);
      \draw[blue, line width=1.2pt] (42,1.348317) -- (42,1.770815);
      \draw[blue, line width=1.2pt] (43,1.212896) -- (43,1.773239);
      \draw[blue, line width=1.2pt] (44,1.164930) -- (44,1.774805);
      \draw[blue, line width=1.2pt] (45,1.119690) -- (45,1.775604);
      \draw[blue, line width=1.2pt] (46,1.323772) -- (46,1.775714);
      \draw[blue, line width=1.2pt] (47,1.256922) -- (47,1.778311);
      \draw[blue, line width=1.2pt] (48,1.253933) -- (48,1.777156);
      \draw[blue, line width=1.2pt] (49,1.291545) -- (49,1.778426);
      \draw[blue, line width=1.2pt] (50,1.148341) -- (50,1.779081);
      \draw[blue, line width=1.2pt] (51,1.128461) -- (51,1.781925);
      \draw[blue, line width=1.2pt] (52,1.157403) -- (52,1.781441);
      \draw[blue, line width=1.2pt] (53,1.197369) -- (53,1.783096);
      \draw[blue, line width=1.2pt] (54,1.368388) -- (54,1.784196);
      \draw[blue, line width=1.2pt] (55,1.046849) -- (55,1.784791);
      \draw[blue, line width=1.2pt] (56,0.980753) -- (56,1.784923);
      \draw[blue, line width=1.2pt] (57,1.106101) -- (57,1.786957);
      \draw[blue, line width=1.2pt] (58,1.161383) -- (58,1.786220);
      \draw[blue, line width=1.2pt] (59,1.257757) -- (59,1.787339);
      \draw[blue, line width=1.2pt] (60,1.250113) -- (60,1.788027);
      \draw[blue, line width=1.2pt] (61,1.232093) -- (61,1.788318);
      \draw[blue, line width=1.2pt] (62,1.288437) -- (62,1.790292);
      \draw[blue, line width=1.2pt] (63,1.155891) -- (63,1.789832);
      \draw[blue, line width=1.2pt] (64,1.111328) -- (64,1.791016);
      \draw[blue, line width=1.2pt] (65,1.205999) -- (65,1.791825);
      \draw[blue, line width=1.2pt] (66,1.117149) -- (66,1.792287);
      \draw[blue, line width=1.2pt] (67,1.221694) -- (67,1.792426);
      \draw[blue, line width=1.2pt] (68,1.191278) -- (68,1.794050);
      \draw[blue, line width=1.2pt] (69,1.074749) -- (69,1.793575);
      \draw[blue, line width=1.2pt] (70,1.166202) -- (70,1.794550);
      \draw[blue, line width=1.2pt] (71,1.188453) -- (71,1.795216);
      \draw[blue, line width=1.2pt] (72,1.163780) -- (72,1.795593);
      \draw[blue, line width=1.2pt] (73,1.308296) -- (73,1.797303);
      \draw[blue, line width=1.2pt] (74,1.135771) -- (74,1.797126);
      \draw[blue, line width=1.2pt] (75,0.989963) -- (75,1.798254);
      \draw[blue, line width=1.2pt] (76,1.036899) -- (76,1.797593);
      \draw[blue, line width=1.2pt] (77,1.073005) -- (77,1.798209);
      \draw[blue, line width=1.2pt] (78,1.139535) -- (78,1.800030);
      \draw[blue, line width=1.2pt] (79,1.186327) -- (79,1.800141);
      \draw[blue, line width=1.2pt] (80,1.169743) -- (80,1.800035);
      \draw[blue, line width=1.2pt] (81,1.157750) -- (81,1.801097);
      \draw[blue, line width=1.2pt] (82,1.225521) -- (82,1.801919);
      \draw[blue, line width=1.2pt] (83,1.149219) -- (83,1.801192);
      \draw[blue, line width=1.2pt] (84,1.165128) -- (84,1.802895);
      \draw[blue, line width=1.2pt] (85,1.039990) -- (85,1.803075);
      \draw[blue, line width=1.2pt] (86,1.077074) -- (86,1.803064);
      \draw[blue, line width=1.2pt] (87,1.061022) -- (87,1.804107);
      \draw[blue, line width=1.2pt] (88,1.210156) -- (88,1.803727);
      \draw[blue, line width=1.2pt] (89,1.361323) -- (89,1.804378);
      \draw[blue, line width=1.2pt] (90,1.326985) -- (90,1.804841);
      \draw[blue, line width=1.2pt] (91,0.851299) -- (91,1.805123);
      \draw[blue, line width=1.2pt] (92,0.959846) -- (92,1.806369);
      \draw[blue, line width=1.2pt] (93,1.029147) -- (93,1.806303);
      \draw[blue, line width=1.2pt] (94,1.044588) -- (94,1.806084);
      \draw[blue, line width=1.2pt] (95,1.142616) -- (95,1.806802);
      \draw[blue, line width=1.2pt] (96,1.164146) -- (96,1.807349);
      \draw[blue, line width=1.2pt] (97,1.101180) -- (97,1.807735);
      \draw[blue, line width=1.2pt] (98,1.177137) -- (98,1.807967);
      \draw[blue, line width=1.2pt] (99,1.280154) -- (99,1.808053);
      \draw[blue, line width=1.2pt] (100,1.169000) -- (100,1.809000);
    \end{tikzpicture}
    \caption{The known bounds on $b_k$ for $k \le 100$. Restricting to $k\le 100$ may give a false impression: the lower bound depends on the distribution of prime powers, and it is possible that $\liminf_k b_k = -\infty$ and $\limsup_k b_k = 0$.}
\label{fig:b_k values and bounds}
\end{figure}

In addition to providing explicit bounds, our approach is also flexible enough to address a number of variations, as we discuss in the rest of the introduction. In particular, it naturally handles the generalization from Sidon sets to $\gamma$-Golomb rulers. Also, we are able to give quantitative results on the distribution of nearly densest-possible $\gamma$-Golomb rulers (including Sidon sets), as well an explicit upper bound on the largest gap between elements that can occur in a dense Golomb ruler.

\subsection{\texorpdfstring{$\gamma$-Golomb rulers}{gamma-Golomb rulers}}
\begin{definition}[$\gamma$-Golomb ruler]
For $\gamma\ge 1$, a $\gamma$-Golomb ruler is a set $\cA$ of integers with the property that for each nonzero $n$, the equation $n=a-b$ has at most $\gamma$ solutions $(a,b)\in \cA\times \cA$. 
\end{definition}
Trivially, $1$-Golomb rulers are Sidon sets. The interval $\{1,2,\dots,k\}$ is a minimum-diameter $\gamma$-Golomb ruler if $k\le \gamma+1$. Some work, e.g.~\cites{2025.Carter&Hunter&Obryant,2023.Balogh&Furedi&Roy}, refer to $\gamma$-Golomb rulers as $\gamma$-thin Sidon sets.

Despite the naturalness of the definition, analysis of the $\gamma>1$ case has been relatively sparse. Atkinson and Hassenklover~\cite{1984.Atkinson&Hassenklover} introduced $\gamma$-Golomb rulers in 1984, with more computational work presented in 1986~\cite{1986.Atkinson&Santoro&Urrutia}. Interest was dormant until the 2015 work~\cite{2015.Caicedo&Martos&Trujillo} of Caicedo--Martos--Trujillo, who proved 
\[ \diam(\cG)\ge \frac1\gamma k^2+o(k^2) \]
for a $\gamma$-Golomb ruler $\cG$ (we give essentially the same result and proof, but in our language, as Theorem~\ref{thm:ET} below). They also proved that there are $\gamma$-Golomb rulers $\cG$ with $k$ arbitrarily large and 
\[ \diam(\cG)=\frac{1}{\gamma}k^2+O(1). \]
More constructions of dense rulers were given in~\cite{2021.Ojeda&Urbano&Solarte} in 2021 and by Gupta--O'Bryant in~\cite{2026.Gupta&Obryant} in 2026.

Balogh--F\"uredi--Roy~\cite{2023.Balogh&Furedi&Roy} note that their technique gives $\diam(\cA) \ge \frac1\gamma (k^2-(2-\epsilon)k^{3/2})+O(1)$, for some $\epsilon=\epsilon(\gamma)>0$ depending on $\gamma$. In~\cite{2025.Carter&Hunter&Obryant}, the explicit value $\epsilon(g)\ge \frac{1}{50\gamma^2}$ is given. In this work, we give $\epsilon>0.11$, the first value independent of $\gamma$.

\begin{corollary}[The minimum diameter of a $\gamma$-Golomb ruler]
    \label{cor:Golomb}
    Suppose that $k\ge\max\{483,\gamma+1\}$, and that $\cA$ is a $\gamma$-Golomb ruler with $k$ elements. Then
    \begin{align*}
        \diam(\cA) &\ge \frac{k^2}{\gamma}\left(\frac{1}{1+0.94304\,(k-\gamma)k^{-3/2}} - \frac{0.94304}{\sqrt{k}}\right)-1 \\
        &\ge\frac{1}{\gamma}\left(k^2 - 1.88607 \, k^{3/2}\right).
    \end{align*}
    If $\cA\subseteq[1,N]$, then
    \[ k \le \sqrt{\gamma N} + 0.94304 \, \sqrt[4]{\gamma N}+ 0.44466.\]
\end{corollary}

\subsection{On the distribution of \texorpdfstring{$\gamma$}{γ}-Golomb rulers}

The method that gives Corollary~\ref{cor:Golomb} can be modified to reveal some distribution information on $\gamma$-Golomb rulers that are within $O(k^{3/2})$ of optimal. One should expect that Golomb rulers are uniformly distributed, i.e., if $\{a_1<\dots<a_k\}$ is a $\gamma$-Golomb ruler with $a_1=0$, then one expects that $a_\ell \approx \frac{\ell}{k} a_k$. Erd\H{o}s and Freud~\cite{1991.Erdos&Freud} proved that roughly optimal Sidon sets are indeed uniformly distributed in this sense, and S.\ W.\ Graham~\cite{1996.Graham} gave a more quantitative, though still not explicit, version. In 2000, Cilleruelo~\cite{2000.Cilleruelo-b} gave the following fully explicit theorem:
\begin{theorem*}[Cilleruelo]\label{thm:Cilleruelo}
    Let $\cA\subseteq[N]$ be a Sidon set with $N^{1/2}-L$ elements. Then any interval $I$ of length $\delta N$ contains $\delta |\cA|+E_I$ elements of $A$, with
    \[|E_I| \le 52 N^{1/4} (1+\delta^{1/2}N^{1/8})(1+L_+^{1/2}N^{-1/8}), \quad L_+= \max\{0,L\}.\]
\end{theorem*}
Note that as the Sidon set because more sparse, i.e., as $L$ increases, the control weakens. Also, as the length of the interval shrinks, i.e., as $\delta$ shrinks, the control strengthens.

We give a result along these lines in Corollary~\ref{cor:ud}. Our result is not fully explicit, in that we have a $O(|\cA|^{3/2})$ term, but there is no theoretical barrier to making our bound explicit. In that regime, however, our asymptotic constant is much sharper that Cilleruelo's. Further, we express the property as a bound on
    \[\abs*{a_\ell -\frac{\ell}{k} a_k}\] 
(assuming $a_1=0$), rather than as bounds on $\abs*{|\cA \cap I| - \frac{\lambda(I)}{\diam(\cA)}|\cA|}$, where $\lambda(I)$ is the length of the interval $I$. And finally, our result applies to $\gamma$-Golomb rulers, not only Sidon sets.

\begin{corollary}
\label{cor:ud}
  Let $b\in\RR$, $\gamma\in\NN$, $\delta\in(0,\frac12]$ be fixed, and let $\cA =\{0=a_1<a_2<\dots<a_k\}$ be a $\gamma$-Golomb ruler with $k$ elements and $\diam(\cA) \le (k^2 - bk^{3/2})/\gamma$. Let $\ell$ be an integer with $1\le \ell\le k$ and $\min\{\ell,k-\ell\}\le \delta k$. Then
  \[\abs*{a_\ell -\frac{\ell}{k} a_k} < \sqrt{\delta(1-\delta)(2.668-b)} \; \frac{k^{7/4}}{\gamma}+ O_{\gamma,b,\delta}(k^{3/2}).\]
\end{corollary}
The constant $2.668$ here is $\sqrt{2}$ times our bound $1.88607$ on $\overline{b_\infty}$.

Another distribution problem that we can address is that of the largest gap in a Golomb ruler. While $a_\ell$ could possibly be $\Omega(k^{7/4})$ away from its expected position, the largest size of $a_{\ell+1}-a_\ell$ must actually be at most $O(k^{3/2})$:
\begin{corollary}
\label{cor:largest gap}
  Let $b\in\RR$, and suppose that $k\ge \max\{9,2\gamma+2\}$. If $\cA=\{a_1<\dots<a_k\}$ is a $\gamma$-Golomb ruler with $\diam(\cA)\le (k^2-b k^{3/2})/\gamma$ then
  \[\max_{1\le \ell < k} (a_{\ell+1}-a_\ell) \le (4-b)\frac{k^{3/2}}{\gamma}+2. \]
\end{corollary}
The constant $4-b$ can be slightly improved with our methods, though doing so involves quite a bit more calculation, so we do not write out a full proof. In the case of Sidon sets (that is, $\gamma=1$), Cilleruelo~\cite{2000.Cilleruelo-b} gave a result of this form without an explicit constant, and later~\cite{2010.Cilleruelo} claimed the stronger bound $\max_{1\le\ell<k}(a_{\ell+1}-a_\ell)\le (2\sqrt{2}-b)k^{3/2}+O(k)$ with an incorrect proof. We suspect that the latter result is still true, and that a variant of our argument could extend that result to $\gamma$-Golomb rulers. This is discussed further at the end of Section~\ref{sec:applications}.

\subsection{Structure of the paper}

In Section~\ref{sec:Theorems}, we explain the framework we used to obtain our main theorems, which are stated in terms of ``WT-triples''. Plugging explicit WT-triple values into our main theorems gives the various corollaries listed earlier in the introduction. In that section, we also give some preliminary results that will be required for the proofs.

In Section~\ref{sec:BIGproof}, we prove our first main result, on the diameter of Golomb rulers. Then, in Section~\ref{sec:applications}, we modify our method to prove our results about the distribution of Golomb rulers.

In Section~\ref{sec:lowerbounds}, we prove some \emph{lower} bounds on WT-triples, which represent limits to what our methods can prove. In Section~\ref{sec:optimization}, we then turn to the computational problem of finding good WT-triples. Finally, in Section~\ref{sec:questions}, we give some open questions.

\subsection{Tool and computational resource disclosure}
This work grew out of efforts to use \pn{AlphaEvolve}~\cite{AlphaEvolve} to improve the solution of the optimization problem arising in~\cite{2025.Carter&Hunter&Obryant}. Under the pressure of reformulating the problem in a manner appealing to \pn{AlphaEvolve}, and while interacting with \pn{AlphaEvolve}, we were led to the new theoretical framework that underpins this work. This framework then resulted in an optimization problem that can be solved with standard optimization software. Our final optimization and verification code is in \pn{Python}, and it as well as all data relevant to this paper is available at \url{https://github.com/dcartermath/sidon-error}. This code is completely human-written, and the relevant details are described in length in Section~\ref{sec:optimization}. 

Section~\ref{sec:optimization}, describing the computation in detail, was also given to \pn{Claude Opus 4.8}, initiating a vibe coding implementation in \pn{Julia}, whose output was checked with human-coded verification in \pn{Mathematica 15}. The results were roughly the same quality, both in terms of computational speed and the consequent bounds.

\pn{Mathematica 15} was used to generate some images in this paper, and \pn{Claude Sonnet} was used to draw Figure~\ref{fig:b_k values and bounds} in \pn{Ti\emph{k}Z}.

\section{Statements of Theorems} \label{sec:Theorems}
All of our theorems are stated in terms of WT-triples (``WT'' stands for ``Window and Table''), and all of our results rely upon their accurate computation. We define WT-triples in Subsection~\ref{subsec:WT-triples}, and discuss the rigorous computation of them in Section~\ref{sec:optimization}. Each of the corollaries given in Section~\ref{sec:introduction} is obtained by straightforward substitution of a particular WT-triple into one of our theorems, except Corollary~\ref{cor:Sidon}, which uses some additional computation.

\subsection{Notation and results from the literature} \label{subsec:setup}
First, we fix some notation. A random function $f$ consists of deterministic functions $f_1,\dots,f_K$ and nonnegative real numbers $p_1,\dots,p_K$ with $\sum p_i =1$, where the probability that $f=f_i$ is $p_i$. If $K=1$ we say that $f$ is \emph{deterministic}. The functions $f_i$ are called the \emph{atoms} of $f$. We have
    \[\EV{f(x)}= \EV{f}(x) = \sum_{i=1}^K p_i f_i(x), \qquad\text{and} \qquad \EV{\norm{f}_2^2} = \sum_{i=1}^K p_i\left(\int_{\RR} f_i(x)^2\,dx\right).\]
Different random functions are not assumed to be independent. That is, by ``random functions $ w, t$'', we mean that there are two finite lists of functions $w_1,\dots,w_K$, $t_1,\dots,t_K$, and a list of probabilities $p_1,\dots,p_K$, so that $\Prob{w=w_i\text{ and } t=t_i} = p_i$.

All of our functions will be bounded with compact support, so issues of convergence are not relevant. We denote the reverse of a function with a tilde, i.e., $\tilde f(x) \coloneqq f(-x)$. For any pair of functions, we have the inner product
    \[\langle f,g\rangle \coloneqq \int_{\RR} f(x)g(x)\,dx.\]
We denote convolution with an asterisk:
    \[f\ast g (x) \coloneqq \int_{\RR} f(y) g(x-y)\,dy.\]
which is commutative, associative, and satisfies the Adjoint Formula
    \[\langle f\ast \tilde g, h\rangle = \langle f, h\ast g \rangle.\]
Further, by Fubini's Theorem
    \[\int_{\RR} f \ast g(x) \,dx = \left(\int_{\RR}f(x)\,dx\right)\left(\int_{\RR}g(x)\,dx\right),\]
and in particular
    \[\int_{\RR} g \ast \tilde g(x) \,dx = \left(\int_{\RR}g(x)\,dx\right)^2.\]
Moreover, for all $x$ the autocorrelation $w\ast \tilde w$ is majorized by the value at 0. To see this, let $h(t)=g(x+t)$ and use Cauchy--Schwarz--Bunyakovsky Inequality\footnote{Cauchy gave the inequality for finite sums in 1821, and this was extended to integrals by his former student Bunyakovsky in 1859. Schwarz extended the inequality to inner product spaces in 1888.} (hereafter \emph{CSB}) as follows:
    \[ |g\ast \tilde g (x)| 
    = |\langle g, h\rangle|
    \le \norm{g}_2 \cdot \norm{h}_2 = \norm{g}_2^2 = g\ast \tilde g(0).\]

\subsection{WT-triples}
\label{subsec:WT-triples}
WT-triples are defined below. As a rule, we want WT-triples for which $\sqrt{\alpha\beta}$ is as small as possible. For explicit results, we want $\upsilon\sqrt{\alpha/\beta}$ to also be sufficiently small.

\begin{definition}[WT-triple]
\label{def:WT-triple}
    We call a triple of positive real numbers $(\upsilon ,\alpha,\beta)$ a \emph{WT-triple} if there is a real number $\sigma \ge 2\upsilon$ and a pair of random integrable functions $w,t$, supported on $[0,\upsilon]$ and $[-\upsilon,\sigma]$ respectively, and 
    \begin{enumerate}[label=({\roman*}),noitemsep]
        \item $t(x)=1$ for $x \in (0,\sigma-\upsilon)$.\label{cond: t flat}
        \item $\int w =1$.\label{cond: w normalized}
    \end{enumerate}
    \begin{enumerate}[resume,label=({\roman*}),noitemsep]
        \item $\EV{ \norm{t}_2^2} \le \sigma + \alpha$.\label{cond: def alpha}
        \item $\EV{\norm{w}_2^2} \le \beta$. \label{cond: def beta}
    \end{enumerate}
    \begin{enumerate}[resume,label=({\roman*}),noitemsep]
    \item $\EV{w\ast \tilde w(x)} \ge 0$ for $x\in \RR$. \label{cond: w positive}
    \item $\EV{ w \ast t(x)} \ge 1$ for $x\in (0,\sigma)$. \label{cond: t*w >1}
    \end{enumerate}
    We say that $(\sigma;w,t)$ \emph{witnesses} the WT-triple $(\upsilon ,\alpha,\beta)$. 

    If $(\upsilon,\alpha,\beta)$ is a WT-triple with witness $(\sigma,w,t)$ whose window $w$ and table $t$ are nonnegative, we call $(\upsilon,\alpha,\beta)$ a \emph{strong WT-triple}.
\end{definition}

The computation of WT-triples is a substantive problem that we address with care in Section~\ref{sec:optimization}. The WT-triples that we actually use are stated in Lemma~\ref{lem:WT triples} and in a table in Table~\ref{tab:WT examples} for the readers' convenience. They are proved by simply exhibiting a witness $(\sigma;w,t)$ that satisfies the definition, and so Lemma~\ref{lem:WT triples} is logically independent of the theorems. The actual witnesses are depicted in Appendix~\ref{sec:special triples}, and they are available in the GitHub repo \url{https://github.com/dcartermath/sidon-error} in infinite precision.

\begin{lemma}[WT-triples]
\label{lem:WT triples}
    The following are strong WT-triples: 
    \[ (1,1,1), \quad(3,1.630563,0.547610), \quad(10,0.944814,0.941259). \] 
\end{lemma}

\begin{table}
    \centering
    \begin{tabular}{ccc|cc|l}
        $\upsilon$ & $\alpha$ & $\beta$ & $2\sqrt{\alpha\beta}$ & $4(1+\upsilon\sqrt{\frac \beta\alpha})^2$ & Comment \\ \hline
        1 & 1 & 1 & 2 & 16 & Erd\H{o}s--Tur\'an's result, for Corollary~\ref{cor:largest gap}\\
        % 2 & 0.943972 & 0.949864 & 1.8939 & 36.1 & smallest deterministic $2\sqrt{\alpha\beta}$ \\
        %22/13 & 0.9474961 & 0.9474927 & 1.8950 & 28.99 & $k_0<29$ \\
        %10 & 0.923432 & 0.9662 & 1.88915 & 504.4 & smallest $2\sqrt{\alpha\beta}$ \\
        3 & 1.630563 & 0.547610 & 1.88989 & 29.999 & $k_0<30$, for Corollary~\ref{cor:Sidon}\\
        10 & 0.944814 & 0.941259 & 1.88607 & 482.4 & smallest $2\sqrt{\alpha\beta}$, for Corollaries~\ref{cor:asymptotic} and~\ref{cor:ud}
    \end{tabular}
    \caption{Examples of strong WT-triples.}
    \label{tab:WT examples}
\end{table}

\subsection{The theorems}
We are now positioned to state our main theorem, which gives a lower bound on the diameter of a $\gamma$-Golomb ruler (equivalently, an upper bound on the density).

\begin{restatable}[Bound on $\gamma$-Golomb rulers]{theorem}{BigBound}
\label{thm:motivation}
Fix a positive integer $\gamma$. If $\cA$ is a $\gamma$-Golomb ruler with $k\ge \gamma+1$ elements, and $(\upsilon,\alpha,\beta)$ is a WT-triple, and \mbox{$k\ge 4\big(1+\upsilon {\sqrt{\beta/\alpha}} \big)^2$}, then
\begin{equation}\label{eq:explicit}
    \diam(\cA) \ge \frac{k^2}{\gamma} \left(\frac{1}{1 + (k-\gamma)k^{-3/2}\sqrt{\alpha\beta}}-\frac{\sqrt{\alpha\beta}}{ k^{1/2}}\right) -1.
\end{equation}
In particular, if $\alpha\beta\le1$ also, then
\begin{equation}\label{eq:puiseux}
\begin{split}
    \diam(\cA) &\ge \frac 1\gamma \cdot k^2 -\frac{2\sqrt{\alpha\beta}}{\gamma} \cdot k^{3/2} +\frac{\alpha\beta}{\gamma} \cdot k \\
    &\qquad + \frac{\sqrt{\alpha\beta}(\gamma-\alpha\beta)}{\gamma}\cdot \sqrt{k} + \frac{\alpha\beta(\alpha\beta-2\gamma)-\gamma}{\gamma}.
    %- O(k^{-1/2}).
\end{split}
\end{equation}
If $\cA\subseteq [1,N]$ (so $N \ge \diam(\cA)-1$), we have
\begin{equation}\label{eq:inverse puiseux}
k \le \sqrt{\gamma N} + \sqrt{\alpha\beta}\sqrt[4]{\gamma N}+\frac{\alpha\beta}{2}+O(N^{-1/4}).
\end{equation}
\end{restatable}

For the reader who wants the shortest route to Corollaries~\ref{cor:Sidon} and~\ref{cor:asymptotic}, our headline results, we quickly give its derivation from Theorem~\ref{thm:motivation} using specific WT-triples.

\begin{proof}[Proof of Corollary~\ref{cor:Sidon}]
Invoke Theorem~\ref{thm:motivation} with $\gamma=1$ and the WT-triple 
\[ (3,1.630563,0.547610) \]
from Lemma~\ref{lem:WT triples}. This gives
\begin{equation}\label{eq:Sidon bound}
\diam(\cA) 
    \ge k^2\left(\frac{1}{1 + 0.944941/\sqrt k}-\frac{0.944941}{ \sqrt k} \right) -1
\end{equation}
for $k\ge 4(1+3\sqrt{0.547610/1.630563})^2 \approx 29.999$. This gives the first inequality of Corollary~\ref{cor:Sidon} for $k\ge 30$. For $1\le k \le 28$, we appeal to the explicit and massive computation of optimal Golomb rulers performed by Distributed.net~\cite{distributed.net}, whose values are archived on Wikipedia~\cite{Wikipedia-Golomb}. That dataset contains the minimum-diameter Sidon set with 28 elements, and its diameter is 585. From this, we know that every 29 element Sidon set has diameter at least 586, so~\eqref{eq:Sidon bound} holds.

For the second inequality, observe for $c>0$ that
\[k^2\left(\frac{1}{1+c/\sqrt k}-\frac{c}{\sqrt k}\right)-1 = k^2-2c k^{3/2}+\frac{c^2 k}{1+c/\sqrt k} \ge k^2-2c k^{3/2}.\]

To verify that  $|\cA|\le N^{1/2}+0.945 N^{1/4}+0.447$, where $N=\diam(\cA)+1$, from Line~\eqref{eq:Sidon bound}, observe that Line~\eqref{eq:Sidon bound} implies an upper bound on $k$ in terms of $N$, and that replacing $k$ with $N^{1/2}+0.945 N^{1/4}+0.447$ gives an inequality that is polynomial in $N^{1/4}$ (after some rearranging and squaring to get rid of the $\sqrt{k}$ terms), and can be algorithmically verified to have no solutions with $N\ge 4$. That the claimed inequality holds with $N=1,2,3$ is trivial, as for these $N$, we have $N^{1/2}+0.945 N^{1/4}+0.447>N$.
\end{proof}

\begin{proof}[Proofs of Corollaries~\ref{cor:asymptotic} and~\ref{cor:Golomb}]
The reasoning for these corollaries is much the same as the proof of Corollary~\ref{cor:Sidon}, except that we use the WT-triple $(10,0.944814,0.941259)$ from Lemma~\ref{lem:WT triples} instead. 

For the last inequality of Corollary~\ref{cor:asymptotic}, we note that $k\ge 483$ implies that
    \[ \gamma N \ge \gamma (1+\diam(\cA)) \ge k^2\left(\frac{1}{1+0.94304\cdot k\cdot k^{-3/2}}-\frac{0.94304}{\sqrt{k}}\right) \ge 2\cdot 10^5.\]
As in the proof of Corollary~\ref{cor:Sidon}, replacing $\diam(\cA)$ with $N-1$, $k-\gamma$ with $k$, and $k$ with $\sqrt{\gamma N}+0.94304 \sqrt[4]{\gamma N}+0.44466$ yields an inequality that is rational in $(\gamma N)^{1/4}$ and is not valid for any $\gamma N\ge 2\cdot 10^5$. Thus $k<\sqrt{\gamma N}+0.94304 \sqrt[4]{\gamma N}+0.44466$.
\end{proof}

Our uniform distribution result is next. Note in what follows that the condition $\diam(\cA) \le (k^2-b k^{3/2})/\gamma$ is not satisfiable for large $k$ if $b> 1.88607$ by Corollary~\ref{cor:Golomb}. There is no requirement for $b$ to be nonnegative.

\begin{restatable}[Dense $\gamma$-Golomb rulers are uniformly distributed]{theorem}{thmud}
\label{thm:ud}
    Let $(\upsilon,\alpha,\beta)$ be a strong WT-triple. 
    Let $b\in\RR,\gamma\in\NN,\delta\in(0,\frac12]$ be fixed, and let $\cA$ be a $\gamma$-Golomb ruler with 
        \[\cA=\{0=a_1<\dots<a_{k}\}, \qquad \diam(\cA) \le (k^2-bk^{3/2})/\gamma.\]
    Let $\ell=\ell(k)$ be an integer with $1\le \ell \le k$ and $\min\{\ell,k-\ell\}\le \delta (k+1)$. Then (as $k\to\infty$)
    \begin{equation}\label{eq:asymp uniform new}
        \left| a_{\ell} - \frac{\ell}{k}\, a_k\right| < k^{7/4} \frac{\sqrt{\delta(1-\delta)}}{\gamma}\sqrt{\sqrt{8\alpha\beta}-b} +O_{\gamma,b,\delta}(k^{3/2}).
    \end{equation}
\end{restatable}

Our final application is to bound the largest gap in a dense $\gamma$-Golomb ruler.

\begin{restatable}[The Largest Gap in a Dense $\gamma$-Golomb ruler]{theorem}{thmlargestgap}
\label{thm:Largest Gap}
  Let $\cA=\{a_1<\dots<a_k\}$ be a $\gamma$-Golomb ruler, and let $(\upsilon,\alpha,\beta)$ be a strong WT-triple. Set $b$ so that $\diam(\cA) \le (k^2-bk^{3/2})/\gamma$. If $k\ge 4\left(1+\upsilon\sqrt{\frac{\beta}{2(\alpha+\upsilon)}}\right)^2$, then
  \[ \max_{1\le i <k} (a_{i+1}-a_i) \le \left(\sqrt{8\beta  (\alpha +\upsilon)}-b\right)\frac{k^{3/2}}{\gamma} +2.\]
\end{restatable}

Corollaries~\ref{cor:ud} and~\ref{cor:largest gap} follow immediately from these theorems by plugging in the two WT-triples $(10,0.944814,0.941259)$ and $(1,1,1)$, respectively.

\subsection{The Caicedo--Martos--Trujillo Bound}
We begin by proving a fully explicit bound motivated by the seminal work of Erd\H{o}s--Tur\'an~\cite{1941.Erdos&Turan}. This preliminary result will be used when proving the stronger claim of Theorem~\ref{thm:motivation}. While our exposition here appears more complicated than necessary, it foreshadows our more powerful technique. Neither the result nor its proof are essentially different from that given in~\cite{2015.Caicedo&Martos&Trujillo}.

\begin{theorem}[Caicedo--Martos--Trujillo~\cite{2015.Caicedo&Martos&Trujillo}]
\label{thm:ET}
If $\cA$ is a $\gamma$-Golomb ruler with $k> \gamma$ elements, then
\[\diam(\cA) \ge \frac1\gamma\left(k^2-2k^{3/2}+k\right)+\sqrt{k}-2 > \frac{k^2-2k^{3/2}}{\gamma}.\]
\end{theorem}

\begin{proof}
If $k\le 3$, the result is trivial.
Let $\cA=\{0=a_1<a_2<\dots < a_k\}$ be a $\gamma$-Golomb ruler with $k > \gamma$ and $k\ge 4$. We define the fattened indicator function of $\cA$ as
  \[F(x) \coloneqq \sum_{a \in \cA} \1_{[a,a+1)}(x),\]
and note that $F\ast \tilde F$ is piecewise linear, with vertices at the integers, and $F\ast \tilde F(n)=d(n)$, where $d(n)$ is the number of ways to write $n$ as a difference of two elements of $\cA$.  Thus, as $\cA$ is a $\gamma$-Golomb ruler, we have
  \[F\ast \tilde F(x) \le (k-\gamma) \Lambda(x) + \gamma,\]
where 
  \[\Lambda(x)\coloneqq \1_{[0,1)} \ast \tilde \1_{[0,1)}(x)=\max\{0,1-|x|\}\] 
is the tent function.

Set $\upsilon \coloneqq (k^{3/2}-k)/\gamma$, which is larger than $1$ since $k\ge 4$ and $k>\gamma$, and let $W \coloneqq \1_{[0,\upsilon )}$, which we call ``the window''. We note that $W\ast \tilde W(x)=\upsilon  \Lambda(x/\upsilon )$, which is supported on $(-\upsilon ,\upsilon )$.
Our control of $F\ast \tilde F$ and the Adjoint Formula allows us to control $F\ast \tilde W$ as follows:
\begin{align*}
  \norm{F \ast \tilde W}_2^2
  & = \langle F \ast \tilde W, F \ast \tilde W\rangle \\
  &= \langle F \ast \tilde F, W \ast \tilde W \rangle\\
  &\le \int_{-\upsilon }^{\upsilon } \left((k-\gamma) \Lambda(x) + \gamma\right) \left(\upsilon  \Lambda(x/\upsilon )\right)\,dx \\
  &\le (k-\gamma)\upsilon  + \gamma \upsilon ^2.
\end{align*}

Set $T \coloneqq \1_{[-\upsilon ,a_k+1)}$, which we call ``the table''. We observe that
  \[  \norm{T}_{2}^2 =a_k +1+\upsilon \]
and that $T \ast W(x) = \upsilon $ for $0\le x \le a_k+1$. As $\supp(F)\subseteq[0,a_k+1]$, we find that
  \begin{equation*}
    \langle F \ast \tilde W, T \rangle
    = \langle F,T\ast W\rangle 
    = \upsilon  \int_{0}^{a_k+1} F(x)\,dx = \upsilon  k.
  \end{equation*}

Using CSB, we now have
\begin{align*}
  \upsilon ^2 k^2 
  &= \langle F \ast \tilde W, T\rangle^2 \\
  &\le \norm{F \ast \tilde W}_2^2 \cdot \norm{T}_2^2 \\
  &\le \left((k-\gamma)\upsilon  + \gamma \upsilon ^2\right)\cdot \left(a_k+1+\upsilon \right),
\end{align*}
and we arrive at
\begin{align*}
    \diam(\cA) = a_k &\ge \frac{\upsilon  k^2}{\gamma \upsilon  +k-\gamma}-1-\upsilon \\
    &=\frac{k^2-2k^{3/2}+k}{\gamma }+\sqrt{k}-2+\frac{\gamma  \left(\sqrt{k}-1\right)}{k^{3/2}-\gamma }.
\end{align*}
As $k^{3/2}\ge k > \gamma$, 
    \[\frac{\gamma  \left(\sqrt{k}-1\right)}{k^{3/2}-\gamma }>0,\]
and so the first inequality of Theorem~\ref{thm:ET} is proved. The second inequality is algebraic, following from $\gamma\ge 1$ and $k\ge 4$.
\end{proof}

The introduction of $W,T$ in the above proof seems unnecessarily complicated, and is done for the sake of foreshadowing our stronger result in the sequel. We have used very little about the window $W$ and table $T$, and one of the main innovations in this work is that we can conduct an efficient search through pairs of functions with the needed properties, which are enumerated in Definition~\ref{def:WT-triple} above.

\subsection{Making WT-triples easier to use}

We also prove two lemmas that make WT-triples and their witnesses easier to use, which will be applied at various points in the paper.
\begin{lemma}\label{lem:big L}
If $(\upsilon,\alpha,\beta)$ is a WT-triple with witness $(\sigma;w,t)$, and $\sigma'\ge 2\upsilon$, then $(\sigma';w,j)$ is also a witness to $(\upsilon,\alpha,\beta)$, where
  \[
    j(x)\coloneqq \begin{cases} 
        t(x) , & -\upsilon \le x \le 0; \\
        1, & 0\le x \le \sigma'-\upsilon; \\
        t(x-\sigma'+\sigma) , & \sigma'-\upsilon\le x\le \sigma';\\
        0 & x\not\in [-\upsilon,\sigma'] .
    \end{cases}
  \] 
\end{lemma}

Loosely speaking, we stretch the middle part of the table $t$, where it is identically $1$, to be whichever size we want, while keeping the edges of $t$. A consequence of this lemma, which is crucial for the proofs of Theorem~\ref{thm:ud} and Theorem~\ref{thm:Largest Gap}, is that we can take witnesses with the same window $w$ and different $\sigma$'s.

\begin{proof}
Let $(\sigma;w,t)$ be a witness to $(\upsilon,\alpha,\beta)$. Since $\sigma'\ge 2\upsilon>0$, the function $j$ is properly defined. We claim that $(\sigma';w,j)$ witnesses $(\upsilon,\alpha,\beta)$. We need to verify that $j$ has the correct support, and satisfies Conditions~\ref{cond: t flat},~\ref{cond: def alpha}, and~\ref{cond: t*w >1} of Definition~\ref{def:WT-triple}.

First, note that $j$ is supported on $[-\upsilon,\sigma']$, as required by Definition~\ref{def:WT-triple}. Condition~\ref{cond: t flat}, that $j(x)=1$ for $x \in [0,\sigma'-\upsilon]$, is immediate from the definition of $j$. Condition~\ref{cond: def alpha}, that $\EV{\norm{j}_2^2} \le \sigma'+\alpha$, is straightforward: as \( t(x)=1 \) for $x\in [0,\sigma-\upsilon]$ and $j(x)=1$ for $x\in [0,\sigma'-\upsilon]$, we have 
  \[\EV{\norm{j}_2^2} = \EV{\norm{t}_2^2} - \sigma+\sigma' \le \sigma'+\alpha.\]

Now, consider Condition~\ref{cond: t*w >1}, which requires that $\EV{w\ast j(x)}\ge 1$ for $x\in (0,\sigma')$. Since $\sigma\ge 2\upsilon$, the functions $j$ and $t$ are identical for inputs in $[-\upsilon,\upsilon]$. Thus, if $0\le x \le \upsilon$ and $0\le y \le \upsilon$, then $x-y \in [-\upsilon,\upsilon]$, and so $j(x-y)=t(x-y)$. Then, as $w$ is supported on $[0,\upsilon]$, we have
  \begin{multline*}
    w\ast j(x) = \int_{\RR}w(y)j(x-y)\,dy = \int_{0}^{\upsilon}w(y) j(x-y)\,dy \\
  =\int_0^\upsilon w(y) t(x-y)\,dy = w\ast t (x).
  \end{multline*}
Likewise, for $\upsilon<x<\sigma'-\upsilon$ and $0\le y\le \upsilon$, we have $j(x)=1$, and so
  \[w\ast j(x) = \int_0^\upsilon w(y) j(x-y)\,dy  
  =\int_0^\upsilon w(y) \cdot 1\,dy = 1,\]
using condition~\ref{cond: w normalized} of Definition~\ref{def:WT-triple}. The $x\ge \sigma'-\upsilon \le x < \sigma'$ case is symmetric with the $0< x\le \upsilon$ case handled above.
\end{proof}

\begin{lemma}\label{lem:tau slide}
    If $(\upsilon,\alpha,\beta)$ is a WT-triple and $\tau>0$, then $(\upsilon/\tau,\alpha/\tau,\beta\tau)$ is also a WT-triple.
\end{lemma}

\begin{proof}
Let $(\sigma;w,t)$ be a witness to $(\upsilon,\alpha,\beta)$.
Let $g(x)=\tau w(\tau x)$ and $h(x)=t(\tau x)$. Then $(\sigma/\tau; g,h)$ witnesses $(\upsilon/\tau, \alpha/\tau,\beta\tau)$.
\end{proof}

We note that Lemma~\ref{lem:tau slide} leaves the quantities $\sqrt{\alpha\beta}$ and $\upsilon \sqrt{ \alpha/\beta}$ unchanged. Both of these invariants play an important role below.

\section{Proof of Theorem~\ref{thm:motivation}}\label{sec:BIGproof}

In this section, we restate and prove our main result, Theorem~\ref{thm:motivation}.

\BigBound*

\begin{proof}%[Proof of Theorem~\ref{thm:motivation}]
Let $(\upsilon,\alpha,\beta)$ be a WT-triple. By Lemma~\ref{lem:tau slide} with $\tau=\gamma\sqrt{\alpha/\beta}$,
    \[(\upsilon',\alpha',\beta') \coloneqq \left(\frac{\upsilon}{\tau},\frac{\alpha}{\tau},\beta\tau \right) = \left(\frac \upsilon{\gamma\sqrt{\alpha/\beta}}, \frac{\sqrt{\alpha\beta}}{\gamma}, \gamma \sqrt{\alpha\beta}\right)\]
is also a WT-triple. This value of $\tau$ has been chosen to optimize
  \[\gamma \alpha'+\frac{\beta'}{\gamma} = 2\sqrt{\alpha\beta}.\]

Let $\cA\subseteq[0,\diam(\cA)]\cap \ZZ$ be a $\gamma$-Golomb ruler with $k>\gamma$ elements. By Theorem~\ref{thm:ET}, we know that $\diam(\cA) \ge (k^2-2k^{3/2})/\gamma$. We assume also that 
  \[ k\ge (2\gamma \upsilon'+2)^2 = 4 \left(1+\upsilon \sqrt{\beta/\alpha}\right)^2,\]
and set
    \[\sigma' \coloneqq  \frac{\diam(\cA)+1}{k^{3/2}} \ge \frac{k^{1/2}-2}{\gamma} \ge 2\upsilon'.\]
As $(\upsilon',\alpha',\beta')$ is a WT-triple, by Lemma~\ref{lem:big L} we can take a witness $(\sigma';w,t)$.

We consider the fat indicator function of $\cA$ (scaled to match $(w,t)$)
    \[f(x) \coloneqq k^{3/2}\,\sum_{n\in\cA} \1_{[n,n+1)}(xk^{3/2})
     = k^{3/2} \sum_{n\in\cA} \1_{[n/k^{3/2},(n+1)/k^{3/2})}(x).\]
Note that $w,t$ are random functions, while $f$ is deterministic. We first observe that $f$ is supported on 
  \[ \supp(f) \subseteq \left[0,\frac{\diam(\cA)+1}{k^{3/2}} \right)=[0,\sigma').\]
We note that $\norm{f}_1 = |\cA| = k$, and that $f\ast \tilde f$ is piecewise linear with vertices at the points $(n/k^{3/2},d(n)\cdot k^{3/2})$, where $d(n)$ is the number of ways to write $n$ as a difference of two elements of $\cA$. Because $\cA$ is a $\gamma$-Golomb ruler, we have
    \begin{equation}\label{eq:Sidon property}
        f\ast\tilde f(x) \le \gamma k^{3/2} + (k-\gamma) k^{3/2}\Lambda(xk^{3/2}),
    \end{equation}
where $\Lambda(x)$ is the ``tent at $0$'' given by $\Lambda(x) = \max\{0,1-|x|\}$.

By CSB, we have
    \[ \langle f \ast \tilde w,t\rangle \le \norm{f\ast \tilde w}_2 \cdot \norm{t}_2.\]
Taking expected values and using CSB again, we arrive at
    \[\EV{\langle f \ast \tilde w,t\rangle}
    \le \EV{\norm{f\ast \tilde w}_2 \cdot \norm{t}_2} \le \EV{\norm{f\ast \tilde w}_2^2}^{1/2} \cdot \EV{\norm{t}_2^2}^{1/2},\]
which we will use in the form
    \begin{equation}\label{eq:CSB result}
        \EV{\langle f \ast \tilde w,t\rangle}^2 \le 
        \EV{\norm{f \ast \tilde w}_2^2} \cdot \EV{\norm{t}_2^2}.
    \end{equation}
We will bound each of these three terms, and then return to this inequality.

\begin{claim}\label{claim: ||f*g||_2}
    \(\displaystyle \EV{\norm{f \ast \tilde w}_2^2} \le \gamma k^{3/2}+(k-\gamma) \beta'. \)
\end{claim}
Having a good bound on $f\ast \tilde f$ allows us to get a good bound on $f\ast \tilde w$ using the Adjoint Formula. In particular,
    \begin{align*}
        \EV{\norm{f \ast \tilde w}_2^2}
    &= \EV{ \langle f \ast \tilde w, f \ast \tilde w \rangle } \\
    &= \EV{ \langle w \ast \tilde w, f \ast \tilde f \rangle }\\
    &= \EV{ \int_{\RR} (w \ast \tilde w(x)) (f \ast \tilde f(x)) \,dx } \\
    &= \int_{\RR} \EV{w \ast \tilde w(x)} (f \ast \tilde f(x))\,dx.
    \end{align*}
Using Condition~\ref{cond: w positive} of being a witness to a WT-triple, we can replace $f\ast\tilde f$ with its upper bound from~\eqref{eq:Sidon property}.
    \begin{multline}
        \int_{\RR} \EV{w \ast \tilde w(x)} (f \ast \tilde f(x))\,dx
        \le \int_{\RR} \EV{w \ast \tilde w(x)} \left( \gamma k^{3/2} + (k-\gamma)k^{3/2} \Lambda(xk^{3/2}) \right)\,dx \\
        = \gamma k^{3/2}\int_{\RR} \EV{w \ast \tilde w(x)}\,dx +(k-\gamma)k^{3/2} \int_{\RR} \EV{w \ast \tilde w(x)} \Lambda(xk^{3/2})\,dx\label{eq:temp1}
    \end{multline}
Now, using Condition~\ref{cond: w normalized}, 
    \[\int_{\RR} \EV{w \ast \tilde w(x)}\,dx = \EV{\int_{\RR} w \ast \tilde w(x) \,dx} = \EV{ \left(\int_{\RR} w\right)^2} = \EV{1^2}=1,\]
which is where we need $\int w =1$ and not merely $\EV{\int w}=1$.
Since $\Lambda(xk^{3/2})=0$ for $x\not\in (-k^{-3/2},k^{-3/2})$, we can bound $\int_{\RR} \EV{w \ast \tilde w(x)} \Lambda(xk^{3/2}) \,dx$ quite sharply. By Condition~\ref{cond: def beta},
    \[\EV{w \ast \tilde w(x)} \le \EV{w \ast \tilde w(0)} = \EV{ \int_{\RR} w(t)^2\,dt} \le \beta'.\]
Thus,
    \[    \int_{\RR} \EV{w \ast \tilde w(x)} \Lambda(xk^{3/2})\,dx
    \le \beta' \int_{\RR} \Lambda(xk^{3/2}) \,dx = \beta' k^{-3/2}.\]
Line~\eqref{eq:temp1} now (using that $k>\gamma$) gives the claimed expression.

\begin{claim}
    \(\displaystyle \EV{ \norm{t}_2^2 } \le (\diam(\cA)+1)\cdot k^{-3/2} +\alpha'\).
\end{claim}

From Condition~\ref{cond: def alpha} and our definition of $\sigma'$, we have
  \[\EV{ \norm{t}_2^2 } \le \sigma'+\alpha' = \frac{\diam(\cA)+1}{k^{3/2}} +\alpha',\]
as claimed.

\begin{claim}
    \(\displaystyle \EV{ \langle f \ast \tilde w, t \rangle} \ge k\)
\end{claim}

The Adjoint Formula gives $\langle f\ast \tilde w,t\rangle = \langle f, t \ast w \rangle$, and $f$ is supported $[0,\sigma')$, where $\EV{ t\ast w (x)}\ge 1$. Thus we have (noting that $\EV{t\ast w}$ is a function)
\begin{equation*}
    \EV{ \langle f \ast \tilde w, t \rangle}
    = \EV{ \langle f , t \ast w\rangle} 
    = \langle f , \EV{t \ast w}\rangle 
    \ge \int f  = k,
\end{equation*}
as claimed.

Now, returning to~\eqref{eq:CSB result}, we have
    \[ k^2 \le \bigg(\gamma k^{3/2}+(k-\gamma)\beta'\bigg)\bigg(\frac{\diam(\cA)+1}{k^{3/2}}+\alpha'\bigg)\]
hence
\begin{align}
    \diam(\cA) &\ge \frac{k^{7/2}}{\gamma k^{3/2}+(k-\gamma)\beta'}-\alpha'k^{3/2}-1 \notag \\
    &= \frac{k^2}{\gamma}\left(\frac{1}{1+(k-\gamma)k^{-3/2}\beta'\gamma^{-1}}-\frac{\alpha'\gamma}{k^{1/2}}\right)-1 \notag\\
    &= \frac{k^2}{\gamma}\left(\frac{1}{1+(k-\gamma)k^{-3/2}\sqrt{\alpha\beta}}-\frac{\sqrt{\alpha\beta}}{k^{1/2}}\right)-1,\label{eq:explicit2}
\end{align}
using the fact $k>\gamma$ so $\gamma k^{3/2}+(k-\gamma)\beta'>0$. This proves~\eqref{eq:explicit}. 

Middle school algebra tells us that the expression on Line~\eqref{eq:explicit2}
is equal to
  \[\frac 1\gamma\left(k^2 - 2c k^{3/2} + c^2 k + c(\gamma-c^2) \sqrt{k} + c^2(c^2-2\gamma)-\gamma +\frac{c^2 E(\gamma ,c,k)}{c (k-\gamma )+k^{3/2}} \right)\]
where $c\coloneqq \sqrt{\alpha\beta}$ and
  \[E(\gamma,c,k) \coloneqq c \gamma  \left(c^2-2 \gamma \right)+\gamma \left(\gamma-c^2\right) \sqrt{k}+c \left(3\gamma-c^2\right) k.\]
For $0<c\le 1$ and $1\le \gamma < k$, we have $E(\gamma,c,k)\ge 0$. This proves~\eqref{eq:puiseux}.

Let
\[ h(x)\coloneq \frac{x^2}{\gamma}\left(\frac{1}{1+(x-\gamma)x^{-3/2}\sqrt{\alpha\beta}}-\frac{\sqrt{\alpha\beta}}{x^{1/2}}\right). \]
Verifying~\eqref{eq:inverse puiseux} is slightly more involved, as $h^{-1}(y)$ does not seem to admit a simple closed form (note that $h^{-1}(y)$ is well-defined for sufficiently large $y$). From the series expansion of $h$, it is clear that $h^{-1}(y)\sim \sqrt{\gamma y}$, so we may write
\[ h^{-1}(y) = \sqrt{\gamma y}(1+z_1) \]
with $z_1=z_1(y)\to 0$ as $y\to\infty$. Substitute this expression for $h^{-1}(y)$ for $x$ in the definition of $h$ and expand as a series in $z_1$ to obtain
\begin{equation}\label{eq:y expansion}
\begin{split}
    y &= h(h^{-1}(y))\\ &= \left(\frac{(\alpha\beta)^2}{\gamma}-\frac{(\alpha\beta)^{\frac32}y^{\frac14}}{\gamma^{\frac34}}+\frac{\alpha\beta y^{\frac12}}{\gamma^{\frac12}}-\frac{2(\alpha\beta)^{\frac12}y^{\frac34}}{\gamma^{\frac14}}-2\alpha\beta+y+(\alpha\beta)^{\frac12}\gamma^{\frac14}y^{\frac14}\right) \\
    &\quad+\left(-\frac{(\alpha\beta)^{\frac32}y^{\frac14}}{2\gamma^{\frac34}}+\frac{\alpha\beta y^{\frac12}}{\gamma^{\frac12}}-\frac{3(\alpha\beta)^{\frac12}y^{\frac34}}{\gamma^{\frac14}}+2y+\frac{1}{2}(\alpha\beta)^{\frac12}\gamma^{\frac14}y^{\frac14}\right)z_1+O(z_1^2).
\end{split}
\end{equation}
Using the fact $z_1(y)\to 0$, if we drop the $O(z_1^2)$ term, solve for $z_1$, and expand as a series in $y$, we obtain
\[ z_1(y)\sim \frac{(\alpha\beta)^{1/2}}{(\gamma y)^{1/4}}, \]
so
\[ h^{-1}(y)=\sqrt{\gamma y} + \sqrt{\alpha\beta}(\gamma y)^{1/4}(1+z_2) \]
with $z_2=z_2(y)\to 0$ as $y\to\infty$. Repeating a similar process gives the estimate $z_2(y)\sim\frac{(\alpha\beta)^{1/2}}{2(\gamma y)^{1/4}}$, so
\[ h^{-1}(y)=\sqrt{\gamma y} + \sqrt{\alpha\beta}(\gamma y)^{1/4}+\frac{\alpha\beta}{2}(1+z_3) \]
with $z_3\to0$. We find $z_3(y)\sim\frac{5\alpha\beta-4\gamma}{4(\alpha\beta)^{1/2}(\gamma y)^{1/4}}$ similarly, which completes the justification for~\eqref{eq:inverse puiseux}.
\end{proof}

\section{On the distribution of \texorpdfstring{$\gamma$}{gamma}-Golomb rulers}
\label{sec:applications}

In the section, we consider two applications of the method of proof of Theorem~\ref{thm:motivation}. We have proven in Corollary~\ref{cor:Golomb} that if $k$ is sufficiently large, then every $\gamma$-Golomb ruler $\cA$ with $k$ elements satisfies
    \(\diam(\cA) \ge \gamma^{-1}(k^2-1.88607 k^{3/2}).\)
The theorems in this section apply to $\gamma$-Golomb rulers that are not much wider; we assume \(\diam(\cA) \le \gamma^{-1}(k^2-b k^{3/2}),\) where $b\in\RR$ can be chosen. Of course, if $b>1.88607$, then the hypothesis is vacuous.

We first give Theorem~\ref{thm:ud}, quantifying the uniform distribution of nearly optimal $\gamma$-Golomb rulers. Then, we give Theorem~\ref{thm:Largest Gap}, bounding the largest gap in a nearly optimal $\gamma$-Golomb ruler. Unlike Theorem~\ref{thm:motivation}, these results require a \emph{strong} WT-triple.

\thmud*

\begin{proof}
By Lemma~\ref{lem:tau slide}, for each $\tau>0$, the triple $(\upsilon',\alpha',\beta') = (\upsilon/\tau,\alpha/\tau,\beta\tau)$ is a WT-triple; we will choose $\tau=\gamma\sqrt{2\alpha/\beta}$ at the end to optimize a constant.
Set
  \begin{align*}
    m_1 &\coloneqq \max\left\{0, 2\upsilon'-\frac{a_\ell+1}{k^{3/2}}\right\},
    &\sigma_1 
    &\coloneqq \frac{a_\ell+1}{k^{3/2}} +m_1 \ge 2\upsilon', \\
    m_2 &\coloneqq \max\left\{0, 2\upsilon'-\frac{a_k+1-a_{\ell+1}}{k^{3/2}} \right\},
    &\sigma_2 
    &\coloneqq \frac{a_k+1-a_{\ell+1}}{k^{3/2}}+m_2 \ge 2\upsilon'.
  \end{align*}
Let $(\sigma_1;w,t_1)$ and $(\sigma_2;w,t_2)$ be witnesses to $(\upsilon',\alpha',\beta')$ with the same nonnegative window and the prescribed $\sigma_1,\sigma_2$.

We define $f=f_1+f_2$, where
\begin{align*}
    f_1(x) 
    &\coloneqq k^{3/2} \sum_{\substack{n\in \cA \\ n\le a_\ell}} \1_{[n,n+1)}(k^{3/2} x), &
    f_2(x) 
    &\coloneqq k^{3/2} \sum_{\substack{n\in \cA \\ n \ge a_{\ell+1}}} \1_{[n,n+1)}(k^{3/2} x).
\end{align*}
Note that $f_1$ is supported on
  \( [0,(a_\ell+1) k^{-3/2}),\) 
which has length at most $\sigma_1$, and $\norm{f_1}_{1} = \ell$. Similarly, $f_2$ is supported on 
  \([(a_{\ell+1} k^{-3/2},(a_k+1)k^{-3/2}),\)
which has length at most $\sigma_2$, and $\norm{f_2}_{1} = k-\ell$.

As in Claim~\ref{claim: ||f*g||_2} in the proof of Theorem~\ref{thm:motivation}, we have
    \begin{equation}\label{eq:claim2.5 here} 
    \EV{\norm{f \ast \tilde w}_2^2} \le \gamma k^{3/2}+(k-\gamma) \beta' \le \gamma k^{3/2}+\beta' k.
    \end{equation}
To get access to distributional information, we write a lower bound on $\EV{\norm{f \ast \tilde w}_2^2}$, which requires that $\tilde w\ge 0$ (instead of merely $w\ast\tilde w\ge 0$). For any nonnegative functions $g,h$, one has $\norm{g+h}_2^2 \ge \norm{g}_2^2+\norm{h}_2^2$. As $(f_1+f_2)\ast \tilde w=f_1\ast\tilde w+f_2\ast\tilde w$, and $\tilde w\ge 0$ so that $f_1\ast\tilde w,f_2\ast\tilde w$ are both nonnegative,
we have
  \begin{align*} 
    \EV{\norm{f\ast \tilde w}_2^2}
    &= \EV{ \norm{(f_1+f_2)\ast \tilde w}_2^2} \\
    &\ge \EV{ \norm{f_1 \ast \tilde w}_2^2 + \norm{f_2 \ast \tilde w}_2^2}\\
    &= \EV{ \norm{f_1 \ast \tilde w}_2^2} + \EV{\norm{f_2 \ast \tilde w}_2^2}.
  \end{align*}
Now, CSB gives
  \[ \langle f_1 \ast \tilde w, t_1 \rangle \le \norm{f_1 \ast \tilde w}_2 \cdot \norm{t_1}_2.\]
Take expected values, square both sides, and use CSB again to find
  \begin{align*}
    \EV{\langle f_1 \ast \tilde w, t_1 \rangle}^2
     &\le \EV{\norm{f_1 \ast \tilde w}_2 \cdot \norm{t_1}_2}^2 \\
     &\le \EV{\norm{f_1 \ast \tilde w}_2^2} \cdot \EV{\norm{t_1}_2^2} \\
     &\le \EV{\norm{f_1 \ast \tilde w}_2^2} \cdot (\sigma_1+\alpha').
  \end{align*}
Here we use the condition $t_1\ge 0$ from $(\upsilon,\alpha,\beta)$ being a strong WT-triple. Using the Adjoint Formula and Condition~\ref{cond: t*w >1} of Definition~\ref{def:WT-triple}, we have
    \[\EV{\langle f_1 \ast \tilde w,t_1\rangle} = \EV{\langle f_1,t_1\ast w\rangle} \ge \ell,\]
and so
  \[\EV{ \norm{ f_1 \ast \tilde w }_2^2} 
  \ge \frac{\ell^2}{\sigma_1+\alpha'} 
  = \frac{\ell^2 k^{3/2}}{a_\ell+1 +(m_1 +\alpha')k^{3/2}}.\]
The same argument works for $f_2$ (though with $t_2$ shifted to match the support of $f_2$), and we thus get
\begin{equation}\label{eq:lower bound}
    \EV{\norm{f\ast \tilde w }_2^2}
    \ge \frac{\ell^2 k^{3/2}}{a_\ell+1 +(m_1 +\alpha')k^{3/2}}
      + \frac{(k-\ell)^2 k^{3/2}}{a_k+1 -a_{\ell+1}+(m_2 +\alpha')k^{3/2}}.
\end{equation}

Comparing the upper bound in \eqref{eq:claim2.5 here} and the lower bound in \eqref{eq:lower bound} gives
  \begin{equation}\label{eq:ud midpoint}
  \frac{\ell^2 k^{1/2}}{a_\ell+1 +(m_1 +\alpha')k^{3/2}}
    + \frac{(k-\ell)^2 k^{1/2}}{a_k+1 -a_{\ell+1}+(m_2 +\alpha')k^{3/2}}
    \le \gamma k^{1/2}+ \beta'.
  \end{equation}
Since $a_{\ell+1}\ge a_\ell+1$, we can replace $1-a_{\ell+1}$ with $-a_\ell$, and both denominators are positive, so we can clear them. We can write inequality~\eqref{eq:ud midpoint} in the form
  \begin{equation} \label{eq:quadratic form}
    A\Big(a_\ell-\frac{\ell}{k}a_k\Big)^2 + \frac{B}{\gamma}\Big(a_\ell-\frac{\ell}{k}a_k\Big)+\frac{C}{\gamma^2} \le 0,
  \end{equation}
with
  \[A = \gamma k^{1/2}+ \beta'  .\]
As $A$ is obviously positive, inequality~\eqref{eq:quadratic form} implies that
  \begin{equation}\label{eq:ud bound 3/4}
  \abs*{ a_\ell - \frac{\ell}{k}a_k }
  \le \frac{1}{\gamma}\left(\frac{|B|}{2A} + \sqrt{\left(\frac{|B|}{2A}\right)^2 -\frac{C}{A}} \,\right).
  \end{equation}

Without loss of generality, we may assume that $\ell \le (k+1)/2$, and we set $d\coloneq\ell/k \le \delta$. Define $c$ so that $a_k=(k^2-ck^{3/2})/\gamma$; we have $1.88607\ge c\ge b$ and
  \begin{align*}
    B &= \left((2 d-1) (\beta' -c \gamma )+\gamma ^2 \left(m_1-m_2\right)\right)k^2\\
        &\qquad+\left(\gamma  (m_1-m_2)-c(2d-1)\right) \beta' k^{3/2} 
        +\gamma ^2 k^{1/2}
        +\beta'  \gamma.
  \end{align*}
As $m_1,m_2,d,\beta',c,\gamma$ are all constants or bounded between constants as $k\to\infty$, we see immediately that $B=O(k^2)$ and so
  \[\frac{|B|}{2A} = \frac{O(k^2)}{2(\gamma k^{1/2}+\beta')}= O(k^{3/2}).\]

It remains to consider $-C/A$. We have
  \begin{align*}
    -C &= \left( d(1-d) \left(\beta' +\gamma  (2 \alpha'  \gamma -c)+\gamma ^2 \left(m_1+m_2\right)\right) \right) k^{4}\\
        &\quad 
        +\bigg( \alpha'  \gamma  \left(\alpha'  \gamma ^2+\beta' -\gamma c \right) 
        + d(1-d)c\left(\gamma  c  - 2 \beta' \right) +\gamma  m_2 \left(\alpha'  \gamma ^2+d (\beta' -c \gamma )\right) \\
        &\qquad\qquad 
          + \gamma  m_1 \left((1-d)(\beta'-c\gamma)+\alpha'  \gamma^2\right) +\gamma^3 m_1 m_2 \bigg) k^{7/2} \\
        &\quad
        + \beta'  \left(\alpha'  \gamma - c d + \gamma  m_1\right) \left(\alpha'  \gamma -c (1-d)+\gamma  m_2\right) k^{3}\\
        &\quad 
        + \gamma ^2 d(1-d) k^{5/2}\\
        &\quad 
        +\gamma  \left(\beta' +\gamma  (\alpha'  \gamma - c (1-d))-\beta'  d+\gamma ^2 m_2\right)  k^{2} \\
        &\quad 
        +\beta'  \gamma  \left(\alpha'  \gamma  - c (1-d)+\gamma  m_2\right) k^{3/2}.
  \end{align*}
The bound on Line~\eqref{eq:ud bound 3/4} is dominated by $\sqrt{-C/A}$, and in particular by the $k^4$ term of $-C$ and the $k^{1/2}$ term of $A$, giving a growth of order $\sqrt{k^4/k^{1/2}}=k^{7/4}$. Hence
\[\left| a_\ell - \frac{\ell}{k}a_k\right| 
  \le \sqrt{\frac{d(1-d) \left(\beta' +\gamma  (2 \alpha'  \gamma -c)+\gamma ^2 \left(m_1+m_2\right)\right)}{\gamma}} \cdot \frac{k^{7/4}}{\gamma} + O(k^{3/2}).\]

By the monotonicity of this bound in $c$, we can replace $c$ with $b$. We have
  \[\frac{\beta'}{\gamma}+2\alpha'\gamma = \frac{\beta\tau}{\gamma}+2\gamma\alpha/\tau,\]
and we choose $\tau>0$ to force $\beta\tau/\gamma = 2\gamma\alpha/\tau$, i.e., we set $\tau \coloneq \gamma \sqrt{2\alpha/\beta}$. With this, we have
  \[\frac{\beta'}{\gamma}+2\alpha'\gamma = \sqrt{8\alpha\beta},\]
so
\begin{equation}\label{eq:before cases}
\left| a_\ell - \frac{\ell}{k}a_k\right| 
  \le \sqrt{\frac{\ell(k-\ell)}{k^2}\left(\sqrt{8\alpha\beta}-b+\gamma(m_1+m_2)\right)} \cdot \frac{k^{7/4}}{\gamma} + O(k^{3/2}).
\end{equation}

We now split into two cases: first, when one of $\ell$ or $k-\ell$ is at most $k^{7/8}$; and second, when both are at least $k^{7/8}$. In the first case, $\frac{\ell(k-\ell)}{k^2} = o(1)$ and the right-hand side of~\eqref{eq:before cases} is $o(k^{7/4})$, and so is stronger than the claimed result~\eqref{eq:asymp uniform new}.

In the second case, from the definitions of $m_1,m_2$, if
\begin{equation}\label{eq:finish}
2\upsilon'k^{3/2} -1 \le a_\ell <a_{\ell+1} \le a_k+1-2\upsilon'k^{3/2}, \end{equation}
then $m_1=m_2=0$ and the result~\eqref{eq:asymp uniform new} holds. Since $\ell\ge k^{7/8}$, we have $a_\ell \geq \frac 1\gamma (\ell^2-1.89\ell^{3/2}) = \Omega(k^{7/4})$ by Theorem~\ref{thm:motivation} applied to the Golomb ruler $\cA_k\cap [0,a_\ell]$, and similarly since $k-\ell\ge k^{7/8}$ we have $a_k-a_{\ell+1}=\Omega(k^{7/4})$. Therefore~\eqref{eq:finish} holds in this case for sufficiently large $k$.

Thus~\eqref{eq:asymp uniform new} holds in all cases, and Theorem~\ref{thm:ud} is proved.
\end{proof}

If we choose $\ell$ in the previous proof so that the largest gap occurs between $a_\ell$ and $a_{\ell+1}$, then we can solve the resulting inequality for the size of this gap. Surprisingly, while $a_\ell$ can be up to $O(k^{7/4})$ away from its expected location, the size of the gap can be only $O(k^{3/2})$.

\thmlargestgap*

\begin{proof}[Proof of Theorem~\ref{thm:Largest Gap}]
Let
  \[g \coloneqq \max_{1\le \ell<k} (a_{\ell+1}-a_\ell),\]
and set $\ell$ so that $g=a_{\ell+1}-a_\ell$.

We proceed as in the proof of Theorem~\ref{thm:ud}, but use $a_{\ell+1}=a_\ell+g$ instead of $a_{\ell+1}\ge a_\ell+1$. 
We start with line~\eqref{eq:ud midpoint}, which guarantees
  \begin{equation}\label{eq:ud midpoint repeate}
  \frac{\ell^2 k^{1/2}}{a_\ell+1 +(m_1 +\alpha')k^{3/2}}
    + \frac{(k-\ell)^2 k^{1/2}}{a_k+1 -a_{\ell+1}+(m_2 +\alpha')k^{3/2}}
    \le \gamma k^{1/2}+ \beta'.
  \end{equation}

Set $D_1,D_2,x,y$ to be the positive quantities
  \begin{align*}
    D_1
    &\coloneqq a_\ell+1+(m_1+\alpha')k^{3/2},
    &x &\coloneqq \ell k^{1/4}, \\    
    D_2
    &\coloneqq a_k+1-(a_\ell+g)+(m_2+\alpha')k^{3/2},
    &y &\coloneqq (k-\ell)k^{1/4}.
  \end{align*}
From Cauchy's Inequality
  \[ (x+y)^2
  =\left(\frac{x}{\sqrt D_1}\cdot\sqrt D_1+\frac{y}{\sqrt D_2}\cdot\sqrt D_2\right)^{2}
  \le \left(\frac{x^2}{D_1}+\frac{y^2}{D_2}\right)(D_1+D_2)
  \]
with
  \[x+y = k^{5/4},\qquad D_1+D_2 = a_k+2-g+(m_1+m_2+2\alpha')k^{3/2} \]
we have
  \[
  \frac{k^{5/2}}{a_k+2-g+(m_1+m_2+2\alpha')k^{3/2}}
  \le \frac{x^2}{D_1}+\frac{y^2}{D_2} \le \gamma k^{1/2}+\beta'.
  \]

Solving this linear inequality for $g$ yields
  \begin{align*}
    g 
    &\le a_k+2+(m_1+m_2+2\alpha')k^{3/2} - \frac{k^{5/2}}{\gamma k^{1/2}+\beta'} \\
    &\le \left((m_1+m_2)\gamma + 2\alpha'\gamma + \frac{\beta'}{\gamma} -b\right) \frac{k^{3/2}}{\gamma}+2,
  \end{align*}
where we have used the elementary $a_k-k^{5/2}/(\gamma k^{1/2}+\beta') \le (\beta'/\gamma-b)k^{3/2}/\gamma$.

If $m_1+m_2\le 2\upsilon'$, then we have
    \begin{equation}\label{eq:g bound prescale}
    g\le \left(2\upsilon'\gamma + 2\alpha'\gamma + \frac{\beta'}{\gamma} -b\right) \frac{k^{3/2}}{\gamma}+2.
    \end{equation}
We set $\tau=\gamma \sqrt{2 (\alpha+\upsilon)/\beta}$, and this becomes
    \begin{equation}\label{eq:g bound}
    g \le \left(\sqrt{8\beta(\alpha+\upsilon)}-b\right) \frac{k^{3/2}}{\gamma}+2,
    \end{equation}
proving Theorem~\ref{thm:Largest Gap} in the case $m_1+m_2\le 2\upsilon'$.

On the other hand, if $m_1+m_2 > 2\upsilon'$, then $m_1>0$ and $m_2>0$ since $m_i \le 2\upsilon'$. Thus
\begin{align*}
    m_1&=2\upsilon'-\frac{a_\ell+1}{k^{3/2}}>0, & m_2=2\upsilon'-\frac{a_k+1-a_{\ell+1}}{k^{3/2}}>0,
\end{align*}
and 
    \[ m_1+m_2 = 4\upsilon'-\frac{2+a_k+g}{k^{3/2}}>2\upsilon'.\]
From Theorem~\ref{thm:ET} we get an upper bound on $a_k$, and that yields
    \[\frac{k^2-2k^{3/2}}{\gamma}
    < 2+a_k+g < 2\upsilon'k^{3/2}.\]
This implies that 
    \[k< 4(1+\upsilon'\gamma)^2=4\left(1+\upsilon\sqrt{\frac{\beta}{2(\alpha+\upsilon)}}\right)^2,\] 
contradicting the hypothesis.
\end{proof}

Being more careful, one can improve~\eqref{eq:g bound prescale} to
\[ g\le \left(\alpha'\gamma+\frac{\beta'}{\gamma}-b+\frac{2c^2+3c\sqrt{\sqrt{8\alpha'\beta'}-b}}{4}\right)\frac{k^{3/2}}{\gamma}+o(k^{3/2}) \]
where $c$ satisfies
\[ 0=4c^3-\gamma(\alpha'+2\upsilon')\left(2c+\sqrt{\sqrt{8\alpha'\beta'}-b}\right), \]
but this calculation is much more involved, so we omit the details. For comparison, for $\alpha'=\beta'=\upsilon'=\gamma=1$ and $b=2$, this gives $g\le 1.954 k^{3/2}+o(k^{3/2})$, compared to $g\le 2k^{3/2}+o(k^{3/2})$ from Theorem~\ref{thm:Largest Gap}. The main idea is that one can split into cases based on whether or not $m_1$ vanishes, and in each case get a better bound than was obtained in Theorem~\ref{thm:Largest Gap}. The worst case occurs when $m_1>0$ and $\ell=\sqrt{ck}\pm o(\sqrt{k})$ with $c$ defined above.

Cilleruelo~\cite{2010.Cilleruelo} (Corollary 6) claims the even better bound $(2\sqrt{2}-b)k^{3/2}+O(k)$ in the case $\gamma=1$. Cilleruelo's result stems from a more careful treatment of Sidon sets contained in a union of intervals (or more generally disjoint arithmetic progressions). However, his Theorem 2, from which Corollary 6 is derived, is incorrect. For example, considering the trivial arithmetic progressions $P_i=\{a_i\}$ (where $\cA=\{a_1,\dots,a_k\}$ is a Sidon set), the theorem reads $k<k^{1/2}+k^{3/4}+1/2$, which is evidently false for $k\ge 9$. Nevertheless, we think that the conclusion of Theorem 2 is probably still true in the case of two arithmetic progressions, just not with that proof, and similarly the claim in Corollary 6 is probably true. We suspect a variant of our argument could save this factor of $\sqrt{2}$, in both Theorems~\ref{thm:ud} and~\ref{thm:Largest Gap}:

\begin{restatable}{conjecture}{conjud}
With the hypotheses of Theorem~\ref{thm:ud}, we have
\[ \left| a_{\ell} - \frac{\ell}{k}\, a_k\right| < k^{7/4} \frac{\sqrt{\delta(1-\delta)}}{\gamma}\sqrt{2\sqrt{\alpha\beta}-b} +O_{\gamma,b,\delta}(k^{3/2}). \]
With the hypotheses of Theorem~\ref{thm:Largest Gap}, we have
\[ \max_{1\le i <k} (a_{i+1}-a_i) \le \left(2\sqrt{\beta  (\alpha +\upsilon)}-b\right)\frac{k^{3/2}}{\gamma} +O_{\gamma,b}(k). \]
\end{restatable}

\section{Lower bounds on WT-triples}
\label{sec:lowerbounds}

In this section we prove lower bounds on various quantities related to WT-triples. These represent limits to what we can prove with our methods. We start with a lower bound on $\alpha+\beta$, which informs part of our optimization procedure for finding WT-triples.

\begin{theorem}\label{thm:lower bound}
Let $f$ be a nonnegative function. Suppose:
\begin{enumerate}[label=(\roman*)]
    \item\label{lower bound 1} $\angleb{f,\1_{(0,\sigma)}}\ge \mu$.
    \item\label{lower bound 2} $\norm{f\ast \tilde{W}}_2^2\le \nu+\norm{W}_2^2$ for all $W$ supported on $[0,\upsilon]$ with $\int W=1$.
\end{enumerate}
Then $\alpha+\beta\ge 2\mu-\nu-\sigma$ for all WT-triples $(\upsilon, \alpha, \beta)$ with witness $(\sigma;w,t)$.
\end{theorem}
\begin{proof}
Let $(\sigma;w,t)$ be a witness for the WT-triple $(\upsilon,\alpha,\beta)$. By~\ref{lower bound 2}, 
\[ \EV{\norm{f\ast \tilde{w}}_2^2}\le \nu+\EV{\norm{w}_2^2} \le \nu+\beta. \]
Therefore by CSB we have
\[ \EV{\angleb{f\ast \tilde{w}, t}}^2 \le \EV{\norm{f\ast \tilde{w}}_2^2}\EV{\norm{t}_2^2}\le (\nu+\beta)(\sigma+\alpha). \]
By the adjoint formula and~\ref{lower bound 1} we have
\[ \EV{\angleb{f\ast \tilde{w}, t}}=\EV{\angleb{f, t\ast w}}\ge \angleb{f, \1_{(0,\sigma)}}\ge \mu. \]
Hence
\[ \mu^2\le (\nu+\beta)(\sigma+\alpha) \le \frac{(\sigma+\nu+\alpha+\beta)^2}{4}, \]
which implies the desired result. Here, we used the inequality 
\[ xy=\frac{(x+y)^2-(x-y)^2}{4}\le \frac{(x+y)^2}{4} \]
with $x=\nu+\beta$, $y=\sigma+\alpha$.
\end{proof}

A natural choice for $f$ witnessing a good lower bound on $\alpha+\beta$ is one that makes the CSB bound~\eqref{eq:CSB result} an equality (or close to an equality), i.e., one naturally guesses that $f$ satisfying $f\ast \tilde{w} \approx t$ witnesses a good lower bound on $\alpha+\beta$ in Theorem~\ref{thm:lower bound}. This turns out to be true, experimentally, at least in the case that $\upsilon$ is sufficiently large. In particular, our computational experiments suggest:
\begin{conjecture}
If $(\upsilon, \alpha, \beta)$ is a WT-triple, then $\alpha+\beta\gtrsim 1.886$.
\end{conjecture}
The meaning of ``$\gtrsim$'' is not made precise; there might be WT-triples with $\alpha+\beta<1.886$, but probably not by much. Since one can apply Lemma~\ref{lem:tau slide} to convert any WT-triple $(\upsilon, \alpha, \beta)$ to a WT-triple $(\upsilon', \alpha', \beta')$ with $\alpha'+\beta'=2\sqrt{\alpha\beta}$, this conjecture is equivalent to:
\begin{restatable}{conjecture}{LowerBound}\label{conj:lowerbound}
If $(\upsilon, \alpha, \beta)$ is a WT-triple, then $2\sqrt{\alpha\beta}\gtrsim 1.886$.
\end{restatable}
Thus, we believe our optimization procedure obtains WT-triples that are very close to the best possible. The $f$ that suggests these conjectures is pictured in Figure~\ref{fig:f}, and it is discussed further in Section~\ref{sec:growing}.

\begin{figure}[htbp!]
    \centering
    \includegraphics[width=0.8\textwidth]{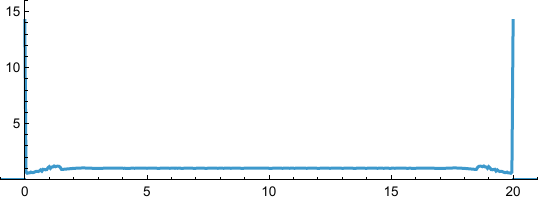}
    \includegraphics[width=0.8\textwidth]{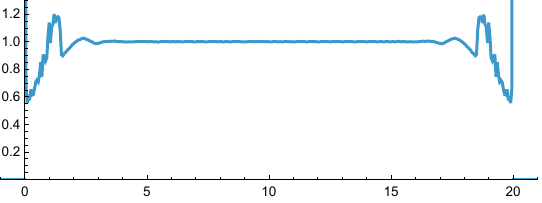}
    \caption{Top: The function $f$ that suggests Conjecture~\ref{conj:lowerbound}. This function satisfies $f\ast \tilde{w}\approx t$ for our best witness $(\sigma;w,t)$. Bottom: Zoomed-in view to show more detail.}
    \label{fig:f}
\end{figure}

Next, we give a lower bound on $\alpha\beta$ itself, which shows that our framework cannot be used to prove $\overline{b_\infty}=0$. Morally speaking, this is just done by plugging a particular $f$ into Theorem~\ref{thm:lower bound} for which we can compute $\mu$ and $\nu$. This $f$ is the indicator of an interval plus a scaled Dirac delta placed at each endpoint. This specific $f$ was inspired by the shape of the experimentally determined $f$ in Figure~\ref{fig:f}.

For this result, we will make the very mild technical assumption that $\EV{t\ast w}$ is continuous. This assumption holds if, for instance, $w$ and $t$ are uniformly bounded, as is the case with our witnesses. Without this assumption, the result nearly still holds, but we must replace the strict inequality with a weak inequality (i.e., $2\sqrt{\alpha\beta}\ge \sqrt{2}$); this may be proved by approximating a general witness $(\sigma; w, t)$ by a sequence of uniformly bounded witnesses. There are a few additional technicalities compared to Theorem~\ref{thm:lower bound}.
\begin{theorem}
If $(\upsilon, \alpha, \beta)$ is a WT-triple with witness $(\sigma; w, t)$ and $\EV{t\ast w}$ is continuous, then $\alpha\beta > 1/2$, so $2\sqrt{\alpha\beta} > \sqrt{2}\approx 1.414$.
\end{theorem}
\begin{proof}
By Lemma~\ref{lem:big L} we may assume $\sigma=2\upsilon$. Let $f(x)\coloneq \1_{(0,2\upsilon)} + \alpha\delta(x) + \alpha\delta(x-2\upsilon)$, where $\delta$ is the Dirac delta. Fix $\varepsilon_1>0$ and let $\varepsilon_2$ be such that $\EV{t\ast w}\ge (1-\varepsilon_2)\1_{(-\varepsilon_1,2\upsilon+\varepsilon_1)}$ pointwise; as $\EV{t\ast w}$ is continuous and at least 1 on $(0,2\upsilon)$, we may take $\varepsilon_2\to 0$ as $\varepsilon_1\to 0$. Then
\begin{align*}
    (2\upsilon + 2\alpha)^2 &= \angleb{f,\1_{(-\varepsilon_1,2\upsilon+\varepsilon_1)}}^2 \\
    &\le (1-\varepsilon_2)^{-2}\EV{\angleb{f,w\ast t}}^2 & \text{\ref{cond: t*w >1}} \\
    &= (1-\varepsilon_2)^{-2}\EV{\angleb{f\ast\tilde{w},t}}^2 & \text{(Adjoint Formula)} \\
    &\le (1-\varepsilon_2)^{-2}\EV{\norm{f\ast\tilde{w}}_2^2}\EV{\norm{t}_2^2} & \text{(CSB)} \\
    &\le (1-\varepsilon_2)^{-2}\EV{\angleb{f\ast\tilde{w},f\ast\tilde{w}}}(\alpha+2\upsilon) & \text{\ref{cond: def alpha}} \\
    &= (1-\varepsilon_2)^{-2}\EV{\angleb{f\ast\tilde{f},w\ast\tilde{w}}}(\alpha+2\upsilon) & \text{(Adjoint Formula).}
\end{align*}
Now, for $x\in (-\upsilon,\upsilon)$, we have
\[ f\ast\tilde{f}(x) = 2\upsilon+2\alpha-\abs{x} + 2\alpha^2\delta(x) \]
by a direct computation. Let $c\in [0,1]$ be such that
\[ \EV{\int_{-\upsilon}^\upsilon (c\upsilon-\abs{x})(w\ast\tilde{w}(x))\;dx}=0. \]
Note that $c$ exists, as this integral is clearly nonpositive when $c=0$ and nonnegative when $c=1$, and it is continuous (in fact, affine) in $c$. Then, using conditions~\ref{cond: w normalized} and~\ref{cond: def beta},
\begin{align*}
    \EV{\angleb{f\ast\tilde{f},w\ast\tilde{w}}} &= \EV{\int_{-\upsilon}^\upsilon \left((2\upsilon+2\alpha - c\upsilon) + (c\upsilon - \abs{x}) + (2\alpha^2\delta(x))\right)(w\ast\tilde{w}(x))\; dx} \\
    &= \EV{(2\upsilon+2\alpha - c\upsilon){\underbrace{\left(\int w\right)}_{=1\text{ by~\ref{cond: w normalized}}}}^2} + \underbrace{\EV{\int_{-\upsilon}^\upsilon (c\upsilon-\abs{x})(w\ast\tilde{w}(x))\;dx}}_{=0\text{ by definition of $c$}} 
     + 2\alpha^2\EV{w\ast\tilde{w}(0)} \\
    &\le (2-c)\upsilon+2\alpha + 2\alpha^2\beta.
\end{align*}
Thus, taking $\varepsilon_1,\varepsilon_2\to 0$, we have
\[ 4(\alpha+\upsilon)^2\le (\alpha+2\upsilon)((2-c)\upsilon+2\alpha+2\alpha^2\beta). \]
Solving for $\alpha\beta$ gives
\begin{align*}
    \alpha\beta&\ge \frac{4(\alpha+\upsilon)^2-(2-c)\upsilon(\alpha+2\upsilon) - 2\alpha(\alpha+2\upsilon)}{2\alpha(\alpha+2\upsilon)} \\
    &= \frac{c\upsilon}{2\alpha} + \frac{\alpha+\upsilon}{\alpha+2\upsilon} \\
    &>1/2
\end{align*}
since $\alpha,\upsilon>0$ and $c\ge 0$, as desired.
\end{proof}
We remark that this bound is unoptimized. A different choice of $f$ could give a better bound. Additionally, the extreme case for this argument is when $c=0$, but this cannot happen for any witness, so the bound can be improved there with more careful analysis.

Finally, we handle the quantity $(\alpha+\upsilon)\beta$ occurring in Theorem~\ref{thm:Largest Gap}.
\begin{theorem}\label{thm:Delta lb}
Let $(\upsilon,\alpha,\beta)$ be a WT-triple. Then $(\alpha+2\upsilon)\beta\ge 3$ and $(\alpha+\upsilon)\beta\ge 2$, and both of these inequalities are tight.
\end{theorem}

\begin{proof}
Let $(\sigma;w,t)$ be a witness to $(\upsilon,\alpha,\beta)$. We may approximate $(\sigma;w,t)$ by a sequence of uniformly bounded witnesses whose associated WT-triples approach $(\upsilon,\alpha,\beta)$, so for the rest of the proof, we will just assume $\EV{t\ast w}$ is continuous.

By Lemma~\ref{lem:big L}, we may assume $\sigma=2\upsilon$. Choose $\varepsilon_1>0$. Define $\varepsilon_2>0$ so that $\EV{t\ast w}\ge (1-\varepsilon_2)\1_{[-\varepsilon_1,2\upsilon+\varepsilon_1]}$ pointwise; since $\EV{t\ast w}$ is continuous and at least 1 on $[0,2\upsilon]$, $\varepsilon_2\to 0$ as $\varepsilon_1\to 0$. Let $f(x)\coloneq\delta(x+\varepsilon_1)+\delta(x-\upsilon)+\delta(x-2\upsilon-\varepsilon_1)$ where $\delta$ is the Dirac delta. Then
\begin{align*}
    9 &= \angleb{f,\1_{(-2\varepsilon_1,2\upsilon+2\varepsilon_1)}}^2 \\
    &\le (1-\varepsilon_2)^{-2}\EV{\angleb{f,t\ast w}}^2 & \text{\ref{cond: t*w >1}} \\
    &= (1-\varepsilon_2)^{-2}\EV{\angleb{f\ast\tilde{w},t}}^2 & \text{(Adjoint Formula)} \\
    &\le (1-\varepsilon_2)^{-2}\EV{\norm{f\ast\tilde{w}}_2^2}\EV{\norm{t}_2^2} & \text{(CSB)} \\
    &= 3(1-\varepsilon_2)^{-2}\EV{\norm{w}_2^2}\EV{\norm{t}_2^2} \\
    &\le 3(1-\varepsilon_2)^{-2}\beta(\alpha+2\upsilon) & \text{\ref{cond: def alpha} and~\ref{cond: def beta}.}
\end{align*}
Sending $\varepsilon_1,\varepsilon_2\to 0$, we find $\beta(\alpha+2\upsilon)\ge 3$, the first inequality we needed to show.

Let
\[ t'(x)\coloneq \begin{cases}
    t(x) & x<0 \\
    t(x+\upsilon) & x \ge 0,
\end{cases} \]
which is formed from $t$ by cutting out the interval $(0,\upsilon)$ (which is constant 1). Note for $x\in (0,\upsilon)$ that
\begin{align*}
    2&\le \EV{w\ast t(x) + w\ast t(x+\upsilon)} \\
    &= \EV{w\ast t'(x) + w\ast\1_{(0,\upsilon)}} \\
    &= \EV{w\ast t'(x)} + 1,
\end{align*}
so $\EV{w\ast t'(x)}\ge 1$ for $x\in (0,\upsilon)$. Let now $f'(x)\coloneq \delta(x+\varepsilon_1) + \delta(x-\upsilon-\varepsilon_1)$. Then
\begin{align*}
    4 &= \angleb{f',\1_{(-2\varepsilon_1,\upsilon+2\varepsilon_1)}}^2 \\
    &\le \dots & \text{(similar steps to before)} \\
    &\le 2(1-\varepsilon_2')^{-2}\EV{\norm{w}_2^2}\EV{\norm{t'}_2^2} \\
    &\le 2(1-\varepsilon_2')^{-2}\beta(\alpha+\upsilon)
\end{align*}
where $\varepsilon_2'$ is defined analogously to $\varepsilon_2$. Hence $\beta(\alpha+\upsilon)\ge 2$.

Both of these inequalities are tight since $(1,1,1)$ is a WT-triple.
\end{proof}

Thus, Corollary~\ref{cor:largest gap} cannot be improved by finding a better WT-triple than $(1,1,1)$. However, we have noted that the statement of Theorem~\ref{thm:Largest Gap} is not the best possible that can be obtained with our methods.

\section{The optimization problem}\label{sec:optimization}

We turn now to the practical problem of computing WT-triples $(\upsilon,\alpha,\beta)$. 
In light of Lemma~\ref{lem:tau slide}, two scale-invariant quantities organize the search:
\begin{align*}
	k_0 &\coloneq 4\big(1+\upsilon\sqrt{\beta/\alpha}\big)^2, &
	b &\coloneq 2\sqrt{\alpha\beta}.
\end{align*}
We are primarily interested in the Pareto frontier of minimal $(k_0,b)$. 
In the setting of Corollary~\ref{cor:Sidon}, where we aim to find an explicit bound, we need $k_0<30$, since the cases $k\le 29$ are handled by the exhaustive computation of Distributed.net~\cite{distributed.net} for $k\le 28$.

Our basic strategy is as follows. 
For each of several values of $\upsilon$ we search for (random) functions $w,t$ satisfying the constraints of Definition~\ref{def:WT-triple}, and then read off $\alpha=\EV{\norm{t}_2^2}-2\upsilon$ and $\beta=\EV{\norm{w}_2^2}$. 
The challenge is, for fixed $\upsilon$, to find $w,t$ making $b$ small. 
This will be done by writing the relevant optimization problem as (essentially) a biconvex optimization problem, which is then solved by a modified proximal alternating minimization scheme.

\subsection{Setup}
\label{sec:setup}

Our functions $w$ and $t$ will be random step functions. For a particular rational number $\upsilon$, we set $\sigma\coloneq 2\upsilon$. The width of the steps in $w$ and $t$ is $\delta\coloneq \upsilon/N$, where $N$ is another chosen parameter. The heights of the steps of $w$ and $t$ are determined by vectors $\vec{g}$ and $\vec{h}$, respectively, which are to be optimized. Finally, $w$ and $t$ are probabilistically chosen, so there are actually $K$ different $\vec{g}^{(i)}$ and $\vec{h}^{(i)}$ vectors which are chosen with probabilities collected in another vector $\vec{p}$ of length $K$. We index vectors starting from 0.

Specifically, let $\vec{g}_i=\angleb{g_0^{(i)},\dots,g_{N-1}^{(i)}}$; then
\[ w_i(x)\coloneq\sum_{j=0}^{N-1}g_j^{(i)}\1_{[j\delta, (j+1)\delta)}(x). \]
Note that $w_i$ is supported on $[0,\upsilon)$. Next, let $\vec{h}^{(i)}=\angleb{h_0^{(i)},\dots,h_{3N-1}^{(i)}}$ and
\[ t_i(x)\coloneq\sum_{j=0}^{3N-1}h_{3N-1-j}^{(i)}\1_{[j\delta,(j+1)\delta)}(x+\upsilon), \]
which is supported on $[-\upsilon,2\upsilon)=[-\upsilon,\sigma)$. Note that the vector $\vec{h}^{(i)}$ is mirrored compared to $t_i$. This is because we care about the convolution $w\ast t$; it is helpful to mirror either $w$ or $t$ so that this convolution may be computed as a dot product of the associated vectors.

We define $w$ and $t$ so that $w=w_i$ and $t=t_i$ with probability $p_i$, where $\vec{p}=\angleb{p_0,\dots,p_{K-1}}$. Finally, we may gather all of the $\vec{g}^{(i)}$ and $\vec{h}^{(i)}$ vectors into matrices $G$ and $H$ of size $K\times N$ and $K\times 3N$, respectively, so that $G_{ij}\coloneq \vec{g}_j^{(i)}$ and $H_{ij}\coloneq \vec{h}_j^{(i)}$, where again we index starting from 0; this is useful for the purposes of vectorizing some of the constraints for faster solving.

\subsection{Constraints}

Let us now convert the conditions in Definition~\ref{def:WT-triple} on $w,t$ into conditions on $G,H$. Note that $w$ and $t$ are supported on $[0,\upsilon]$ and $[-\upsilon,\sigma]$, as required.

\begin{enumerate}[label=(\roman*)]
    \item The condition $t(x)=1$ for $x\in(0,\sigma-\upsilon)$ means that we must enforce
    \begin{equation}\label{eq:condition0}
        H_{ij}=1\qquad\text{for $N\le j\le 2N-1$ and all $i$}.
    \end{equation}
    \item The condition $\int w=1$ translates to the condition
    \begin{equation}\label{eq:condition1}
        \sum_{j=0}^{N-1}g_j^{(i)}\cdot\delta=1\qquad\text{for all $i$}.
    \end{equation}
    \item The condition $\EV{\norm{t}_2^2}\le \sigma+\alpha$ is what defines the value $\alpha$, which contributes to the objective function. We define
    \[ A\coloneq -\sigma + \sum_{i=0}^{K-1}p_i\norm{\vec{h}^{(i)}}_2^2\cdot\delta. \]
    \item Similarly, the condition $\EV{\norm{w}_2^2}\le \beta$ is what defines the value $\beta$, which also contributes to the objective function. We define
    \[ B\coloneq \sum_{i=0}^{K-1}p_i\norm{\vec{g}^{(i)}}_2^2\cdot\delta. \]
    \item Instead of the condition $\EV{w\ast \tilde{w}(x)}\ge 0$ for $x\in\RR$, we instead use the stronger condition $w(x)\ge 0$; that is, we will only find strong WT-triples with our optimization. The condition translates to
    \begin{equation}\label{eq:condition2}
        G_{ij}\ge 0\qquad\text{for all $i,j$}.
    \end{equation}
    \item Finally, the condition $\EV{w\ast t(x)}\ge 1$ for $x\in (0,\sigma)$ is translated as follows. Note that $w\ast t$ is piecewise linear and continuous with knuckles at $\delta j$ where $j$ is an integer. Thus, if $\EV{w\ast t(x)}\ge 1$ at just the points $x=0,\delta,2\delta,\dots, \sigma$, then $\EV{w\ast t(x)}\ge 1$ for all $x\in (0,\sigma)$. Therefore this condition translates to
    \begin{equation}\label{eq:condition3}
        \tr(\diag(\vec{p})G\cdot H[j,j+N-1]^T)\cdot\delta\ge 1\qquad\text{for $0\le j\le 2N$},
    \end{equation}
    where $H[a,b]$ is the submatrix of $H$ between and including columns $a$ and $b$ and $\diag(\vec{p})$ is the $K\times K$ matrix whose diagonal entries are the entries of $\vec{p}$ and whose off-diagonal entries are 0. To see that the left-hand side of~\eqref{eq:condition3} computes $\EV{w\ast t(x)}$, note that the $(i,i)$ entry of $G\cdot H[j,j+N-1]^T$ is equal to $\vec{g}^{(i)}\cdot \vec{h}^{(i)}[j,j+N-1]$, which is $w_i\ast t_i(j\delta)/\delta$ (recalling that $\vec{h}^{(i)}$ is mirrored compared to $t_i$).
\end{enumerate}
Finally, we have the additional constraints
\begin{equation}\label{eq:condition4}
    p_i\ge 0\qquad\text{for all $i$}
\end{equation}
and
\begin{equation}\label{eq:condition5}
    \sum_{i=0}^{K-1}p_i=1,
\end{equation}
so that $\vec{p}$ defines a probability distribution.

We do not optimize $2\sqrt{AB}$ (as might be most natural, as the quantity $b\coloneqq 2\sqrt{\alpha\beta}$ is what we are most interested in), or similar, but instead the sum $A+B$. This is because $\sqrt{AB}$ (or $AB$) is not convex in the variables $\vec{p}$ if one holds $G$ and $H$ constant, but $A+B$ is. Lemma~\ref{lem:tau slide} makes this an inexpensive decision, as we can rescale our witnesses after the fact; optimizing $A+B$ is almost just as good as optimizing $2\sqrt{AB}$.

\subsection{Biconvex formulation}
\label{sec:biconvex}

As it stands, the objective function $A+B$ and the constraints are actually triconvex in the variable sets $\{\vec{p},G,H\}$, meaning that if one holds two of these three variable sets constant, the optimization problem is a convex program in the third variable set. We found much more success by modifying the constraints so that one may optimize $\vec{p}$ and $G$ holding $H$ constant, or optimize $\vec{p}$ and $H$ holding $G$ constant, rather than the natural triconvex formulation. This reformulation may be done as follows.

Define the matrix $G'$ so that $G'_{ij}\coloneq \vec{p}_iG_{ij}$ and $H'$ so that $H'_{ij}\coloneq \vec{p}_iH_{ij}$. Then all of the conditions~\eqref{eq:condition0} through~\eqref{eq:condition3} may be straightforwardly rewritten either in the variable sets $\{G',H\}$ or $\{G,H'\}$; in particular, condition~\eqref{eq:condition3} becomes either
\[ \tr(G'\cdot H[j,j+N-1]^T)\cdot\delta\ge 1 \]
or
\[ \tr(G\cdot H'[j,j+N-1]^T)\cdot\delta\ge 1 \]
which are biconvex in $\{G',H\}$ and $\{G,H'\}$, respectively.

The definition of $A$ using $H'$ instead of $H$ becomes a sum of quadratic-over-linear expressions:
\[ A \coloneqq -\sigma + \sum_{i=0}^{K-1}\frac{\norm{H'_i}_2^2}{p_i}\cdot \delta \]
where $H'_i$ is the $i$th row of $H'$. Likewise,
\[ B \coloneqq \sum_{i=0}^{K-1}\frac{\norm{G'_i}_2^2}{p_i}\cdot \delta \]
is a sum of quadratic-over-linear expressions. These are convex expressions in the variable sets $\{\vec{p},H'\}$ and $\{\vec{p},G'\}$, respectively, while the original definitions are convex in $\vec{p}$ if one holds either $H$ or $G$ constant.

With this reformulation, a convex solver may be used to optimize $\{\vec{p},G'\}$ holding $H$ constant, or optimize $\{\vec{p},H'\}$ holding $G$ constant. Alternating between these two convex programs forms the bulk of our optimization procedure, but there are a few additional complications.

\subsection{Growing and shrinking \texorpdfstring{$K$}{K}}
\label{sec:growing}

As described so far, our optimization procedure is relatively sensitive to the initial choice of $G$ and $H$, and different choices lead to better or worse resulting bounds on $b$. However, there is a way to guess new vectors $\vec{g}$ and $\vec{h}$ such that adding them to the ensemble with positive probability improves the objective value. With this ability, one may start with some $G$ and $H$ with a small value of $K$ ($K=1$ for instance), and after the biconvex optimization stalls, increase the value of $K$ by 1 by adding a new row to $G$ and $H$ according to the rule we will describe. In this way, we obtain consistently good results without having to come up with a good initial choice of $G$ or $H$; we just used a scaled indicator of an interval for the initial value of each row of $G$.

To do this, we compute
\begin{equation}\label{eq:f def}
\tilde{f}\coloneq \sum_{i=0}^{K-1} p_i\cF^{-1}\left[\frac{\cF[\vec{h}^{(i)}]}{\cF[\vec{g}^{(i)}]}\right]
\end{equation}
where $\cF$ is the discrete Fourier transform, computed using SciPy's implementation~\cite{virtanen2020scipy}. Here, the division is done pointwise; one might worry that this could lead to a division by 0, but this never occurs in practice. This vector $\tilde{f}$ is approximately zero outside of an interval of length $2N+1$; we truncate to this interval to obtain vector $\vec{f}$.

Next, we find $\vec{g}$ that maximizes $\norm{\vec{f}\ast \vec{g}\cdot\delta}_2^2\cdot\delta - \norm{\vec{g}}_2^2\cdot\delta$ (where the convolution outputs a new vector of length $3N$) and append $\vec{g}$ as a new row in $G$, and increment $K$. The problem of maximizing this quantity is not a convex problem, but is instead \emph{convex-concave}. This means we are not able to find a global solution efficiently, so we just find a good local solution. Finally, we continue the biconvex optimization by first optimizing $\{\vec{p},H'\}$ holding $G$ constant, noting that this optimization step will add a new row to $H$ and a new entry in $\vec{p}$ as well.

The reason that this works was explained previously in Section~\ref{sec:lowerbounds}. This $f$ is one such that the CSB bound~\eqref{eq:CSB result} nearly an equality, i.e., this choice makes $f\ast \tilde{w}\approx t$.

Finally, it frequently occurs that some window-table pairs in the ensemble are chosen with extremely low probability. We may delete the corresponding entry of $\vec{p}$ and rows of $G$ and $H$ before continuing with the optimization. This has little effect on the reported bound, and it keeps the problems small, which helps the performance of the convex solvers. Through these growing and shrinking processes, the value of $K$ tended to fluctuate between 6 and 11; our final witness happened to use an ensemble of size 6.

\subsection{Rounding and further technical details}
\label{sec:rounding}

We employed two convex solvers for different stages of the optimization procedure. First, we used the solver \pn{CLARABEL}~\cite{Clarabel_2024} for fast, low-precision solves for both the alternating minimization procedure described in Section~\ref{sec:biconvex} and solving the convex-concave program in Section~\ref{sec:growing}. The relevant convex and convex-concave programs were modeled using \pn{CVXPY}~\cite{diamond2016cvxpy,agrawal2018rewriting} and its extension \pn{DCCP}~\cite{shen2016disciplined}.

However, this only gives functions $w,t$ that approximately satisfy the required conditions. In order to find an exact solution, we used the high-precision solver \pn{SDPA-GMP}~\cite{SDPA-GMP} in the last step. Specifically, if we have a high-quality approximate solution $G$ (in floating point), we first round every entry of $G$ to a rational number with fixed denominator, then nudge the values slightly so that condition~\ref{cond: w normalized} is satisfied. Next, we formulate the relevant convex program that optimizes $\{\vec{p},H'\}$ as a semidefinite program (SDP), using the following conversion for quadratic-over-linear expressions to a semidefinite cone constraint:
\[ \frac{\norm{\vec{a}}_2^2}{b}\le c\qquad\Leftrightarrow\qquad\begin{bmatrix}
    c & \vec{a}^T \\
    \vec{a} & bI
\end{bmatrix}\text{ is positive semidefinite}. \]
A similar conversion is found in the MOSEK cookbook~\cite{cookbook}, Section 6.2.5. The other constraints are straightforward to convert to linear and semidefinite cone constraints.

Before solving with \pn{SDPA-GMP}, we make one more modification. Let $\varepsilon$ be the chosen solver tolerance; we used $\varepsilon=10^{-30}$. Then we replace each inequality constraint $X\ge Y$ with the constraint $X\ge (1+\varepsilon)Y$ before solving with \pn{SDPA-GMP}. This makes the solution obtained be extremely close to an exact solution; due to the solver tolerance, the obtained solution will be guaranteed to actually satisfy $X\ge Y$, which is what we require. The cost of this trick is that \pn{SDPA-GMP} will emit an error status as it (rightfully) is unable to determine whether or not the program has a feasible point, and the obtained solution will have a relatively large duality gap compared to the solver tolerance, indicating that the solver failed to optimize to completion. However, the duality gap is still much smaller than what is required for our purposes, and solving all the way to optimality would not visibly affect the reported numbers.

In the obtained (floating point) solution, we convert everything to rational numbers (now with quite large denominators), compute $H_{ij}=H'_{ij}/p_i$, and finally replace $\vec{p}$ by $\vec{p}/\norm{p}_1$ to normalize it, since the equality constraint $\sum p_i=1$ was not modified to counteract the solver tolerance. The resulting rational triple $\{\vec{p},G,H\}$ is then checked to make sure it satisfies all of the requirements exactly.

A flowchart outlining the complete optimization process is drawn in Figure~\ref{fig:flowchart}. The complete technical details, code, and instructions are available in the GitHub~\url{https://github.com/dcartermath/sidon-error}.

We also mention here that the (nonconvex) solver \pn{IPOPT}~\cite{IPOPT} was very helpful in earlier attempts at optimizing WT-triples, though it does not feature in the final procedure. 

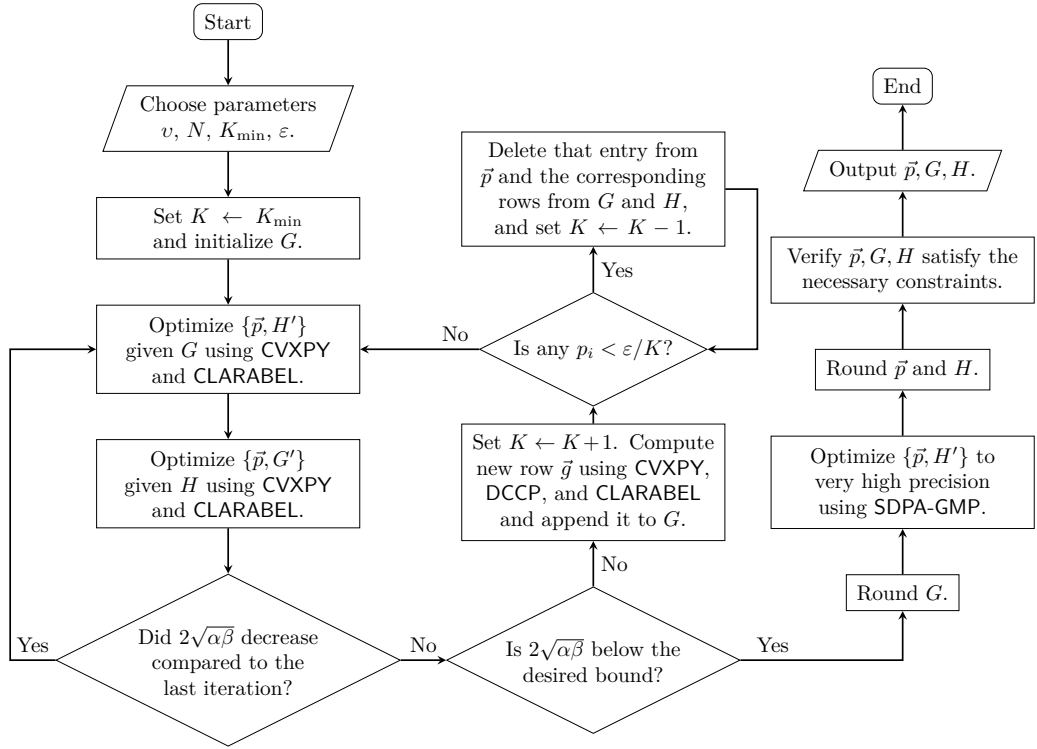
\begin{figure}[htbp!]
    \centering
    \scalebox{0.8}{
    \begin{tikzpicture}[node distance=0.75cm]
        \node (start) [startstop] {Start};
        
        \node (input) [longio, below=of start.south] {Choose parameters $\upsilon$, $N$, $K_{\text{min}}$, $\varepsilon$.};
        
        \node (init) [longprocess, below=of input.south] {Set $K\gets K_{\text{min}}$ and initialize $G$.};
        
        \node (optH) [longprocess, below=of init.south] {Optimize $\{\vec{p},H'\}$ given $G$ using \pn{CVXPY} and \pn{CLARABEL}.};
        
        \node (optG) [longprocess, below=of optH.south] {Optimize $\{\vec{p},G'\}$ given~$H$ using \pn{CVXPY} and \pn{CLARABEL}.};
        
        \node (mainloop) [longdecision, below=of optG.south] {Did $2\sqrt{\alpha\beta}$ decrease compared to the last iteration?};
        
        \coordinate [left=of mainloop.west] (phantom1);
        
        \node (secondloop) [longdecision, right=of mainloop.east] {Is $2\sqrt{\alpha\beta}$ below the desired bound?};
        
        \node (extend) [longprocess, above=of secondloop.north] {Set $K\gets K+1$. Compute new row $\vec{g}$ using \pn{CVXPY}, \pn{DCCP}, and \pn{CLARABEL} and append it to $G$.};

        \path let \p1 = (extend), \p2 = (optH) in node (cleanupdec) [decision] at (\x1,\y2) {Is any $p_i<\varepsilon / K$?};

        \node (cleanup) [longprocess, above=of cleanupdec.north] {Delete that entry from $\vec{p}$ and the corresponding rows from $G$ and $H$, and set $K\gets K-1$.};

        \coordinate [below right=of cleanup.east] (phantom2);

        % \coordinate [above right=of extend.east] (phantom2);

        \node (highacc) [longprocess, right=of extend.east] {Optimize $\{\vec{p},H'\}$ to very high precision using \pn{SDPA-GMP}.};

        \node (roundG) [process, below=of highacc.south] {Round $G$.};

        \node (roundpH) [process, above=of highacc.north] {Round $\vec{p}$ and $H$.};

        \node (verify) [longprocess, above=of roundpH.north] {Verify $\vec{p},G,H$ satisfy the necessary constraints.};

        \node (output) [io, above=of verify.north] {Output $\vec{p},G,H$.};

        \node (end) [startstop, above=of output.north] {End};

        \draw [arrow] (start) -- (input);
        \draw [arrow] (input) -- (init);
        \draw [arrow] (init) -- (optH);
        \draw [arrow] (optH) -- (optG);
        \draw [arrow] (optG) -- (mainloop);
        \draw [arrow] (mainloop) -- node[midway, anchor=south] {Yes} (phantom1) |- (optH);
        \draw [arrow] (mainloop) -- node[midway, anchor=south] {No} (secondloop);
        \draw [arrow] (secondloop) -- node[midway, anchor=west] {No} (extend);
        \draw [arrow] (extend) -- (cleanupdec);
        \draw [arrow] (cleanupdec) -- node[pos=0.2, anchor=south] {No} (optH);
        \draw [arrow] (cleanupdec) -- node[midway, anchor=west] {Yes} (cleanup);
        \draw [arrow] (cleanup) -| (phantom2) |- (cleanupdec);
        \draw [arrow] (secondloop.east) -| node[pos=0.1, anchor=south] {Yes} (roundG);
        \draw [arrow] (roundG) -- (highacc);
        \draw [arrow] (highacc) -- (roundpH);
        \draw [arrow] (roundpH) -- (verify);
        \draw [arrow] (verify) -- (output);
        \draw [arrow] (output) -- (end);
    \end{tikzpicture}
    }
    \caption{The optimization process for finding WT-triples. The processes in the first column are described in greater detail in Sections~\ref{sec:setup} through~\ref{sec:biconvex}, the processes in the second column is described in Section~\ref{sec:growing}, and those in the third column are described in Section~\ref{sec:rounding}.}
    \label{fig:flowchart}
\end{figure}

\section{Unresolved Questions}\label{sec:questions}

\subsection{The nature of WT-triples}

Our best WT-triples are witnessed by random asymmetric functions, and experimentally, restricting to deterministic witnesses or symmetric witnesses yields worse bounds. There is strong evidence for the following, though we lack an intuitive explanation:
\begin{conjecture}
Let $x_1$ be the infimum of $2\sqrt{\alpha\beta}$ over all WT-triples $(\upsilon,\alpha,\beta)$, $x_2$ be the infimum over those triples with deterministic witnesses, and $x_3$ be the infimum over those triples with symmetric witnesses. Then $x_1<x_2$ and $x_1<x_3$.
\end{conjecture}
Note that there are multiple interpretations for what ``symmetric'' should mean; for instance, it could mean that every atom of $w,t$ is symmetric, or that $w$ and $t$ are symmetric in expectation.

Though we obtained a lower bound on $2\sqrt{\alpha\beta}$, there remains a gap between this lower bound and our best known WT-triple. We suspect our examples are close to best-possible:
\LowerBound

Thus, obtaining significantly better bounds on the diameter of Sidon sets must probably be done by another method. Perhaps further improvements will have to come by studying the algebraic properties of dense Sidon sets, instead of the purely analytic approach taken in this paper.

\subsection{The nature of Sidon sets and Golomb rulers}

The shape of the function $f$ in Figure~\ref{fig:f} suggests that extremely high-quality Sidon sets ought to be significantly denser near their ends than in the middle. The best-known constructions do not have this property.
\begin{question}
Are there better constructions of Sidon sets (or Golomb rulers) that put a much higher density of elements near the ends compared to the middle?
\end{question}

We suspect that similar methods to our work in the proof of Theorem~\ref{thm:ud} could give explicit bounds for the maximum size of a $\gamma$-Golomb ruler in a union of intervals, or more generally address the following question:
\begin{question}
Let $X\subseteq[0,1]$ be a measurable set, and denote by $N\cdot X$ the dilation of $X$ by a factor $N$. How large is the largest $\gamma$-Golomb ruler contained in $(N\cdot X)\cap \ZZ$, as $N\to\infty$? 
\end{question}
The literature contains results under the assumption that $X$ is a union of two intervals~\cite{1990.Abbott,2022.Riblet} or an arbitrary number of intervals~\cite{2010.Cilleruelo}, though as we noted in Section~\ref{sec:applications}, the latter result is not correct. We suspect that extending our techniques to the case that $X$ is a union of two intervals would improve some constants in Theorems~\ref{thm:ud} and~\ref{thm:Largest Gap} by a factor of $\sqrt{2}$:

\conjud*

In Theorem~\ref{thm:Largest Gap} we gave an upper bound on the size of the largest gap in a Sidon set. There is no known superlinear lower bound on this quantity, and it may be that the shape of the bound in Theorem~\ref{thm:Largest Gap} is optimal:
\begin{conjecture}
Let $\cA=\{a_1<\dots<a_k\}$ be a $\gamma$-Golomb ruler with $\diam(\cA)\le (k^2-bk^{3/2})/\gamma$. Then there is a constant $c=c(b)>0$ such that
\[ \max_{1\le i<k}(a_{i+1}-a_i)\ge c\frac{k^{3/2}}{\gamma}-o(k^{3/2}). \]
\end{conjecture}

Our methods seem ill-equipped to address questions about algebraic properties of dense Sidon sets. For instance, it ought to be the case that approximately half of the elements in a dense Sidon set are odd and half are even, but our method ``smears out'' the set, so that any information about such detailed features is destroyed.
\begin{problem}
Develop methods which may be used to study algebraic properties of dense Sidon sets.
\end{problem}

\section{Acknowledgments}

The first author is supported by NSF grant DGE-2444107. We thank Bogdan Georgiev, Javier Gomez Serrano, and Adam Zsolt Wagner for their work on AlphaEvolve, and we thank those three as well as Zach Hunter for discussions early on in this research project.

\printbibliography

\clearpage

\appendix

\section{Special WT-triples and witnesses}\label{sec:special triples}

In this appendix, we only use $\sigma=2\upsilon$, and we describe a witness for a WT-triple $(\upsilon,\alpha,\beta)$ by giving the pair $(w,t)$ explicitly, following the convections from Section~\ref{sec:setup}.
Specifically, $w$ and $t$ are random step functions on the uniform grid of width $\delta = \upsilon/N$. The heights of the steps of $w,t$ are given by $K$ vectors $\vec{g}^{(i)}$ (of length $N$, indexed from $0$) and $\vec{h}^{(i)} $ (of length $3N$, indexed from $0$). Additionally, if there are $K$ options for $w,t$, then we give a vector $\vec{p}$ of probabilities. The exact data, given as rational numbers, is available on the GitHub repository \url{https://github.com/dcartermath/sidon-error}.

Note that the verification a witness $(\sigma;w,t)$ actually satisfies the properties needed to witness $(\upsilon,\alpha,\beta)$ is a straightforward check; code for this check is also available in the GitHub repository. The verification code is a straightforward translation from the definition of WT-triple witnesses.

Recall also the derived quantities
    \[ b \coloneqq 2\sqrt{\alpha\beta}, \qquad k_0 \coloneqq 4\bigl(1+\upsilon\sqrt{\beta/\alpha}\bigr)^2, \qquad \Delta \coloneqq \beta(\upsilon+\alpha), \]
which are invariant under the transformation of Lemma~\ref{lem:tau slide}.

\subsection{The Classic WT-triple}
\label{sec:ET triple}
The methods developed in this work began by reformulating the Erd\H{o}s-Tur\'an argument with the following WT-triple, which witnesses $\Delta=2$, the best possible by Theorem~\ref{thm:Delta lb}.
\begin{claim}[Erd\H{o}s-Tur\'an]\label{claim:ET}
The triple $(\upsilon,\alpha,\beta)=(1,1,1)$ is a WT-triple.
\end{claim}

\begin{proof}
Let $(\upsilon,\alpha,\beta)=(1,1,1)$, so that $\sigma=2\upsilon=2$, and consider the (deterministic) functions
$w=\1_{[0,1]}$ and $t=\1_{[-1,2]}$. Then $w$ is supported on $[0,\upsilon]$, and $\int w =1$, and $\norm{w}_2^2=1\le \beta$. Further, $w\ast\tilde w$ is nonnegative because $w\ge 0$. We also have that $t$ is supported on $[-\upsilon,\sigma]$, and $t(x)=1$ for $x\in [0,\sigma-\upsilon]$. Also $\norm{t}_2^2 = 3= \sigma+\alpha$.

Finally,
\[ w\ast t(x) = \begin{cases}
 0, & x \le -1; \\
 x, & -1\le x \le 0\\
 1, & 0 \le x \le 2 \\
 4-x, & 2 \le x \le 3 \\
 0 & x\ge 4;
\end{cases}\]
and in particular $w\ast t(x) =1$ for $x\in [0,\sigma]$. Thus, $(w,t)$ witnesses the WT-triple $(1,1,1)$.
\end{proof}

This witness is strong, and it has $b=2$, $k_0=16$ and $\Delta=2$. Every improvement
recorded below is an improvement in $b$; the interest of the remaining two subsections
is that one keeps $k_0$ below $30$ while the other pushes $b$ as low as we can manage without regard to $k_0$.

\subsection{\texorpdfstring{A WT-triple with small $b$ and $k_0<30$}{A WT-triple with small b and k-naught below 30}}
\label{sec:small k0 triple}
We give a witness to a WT-triple with $\upsilon=3$ whose $k_0$ is just under $30$, and whose $b$ is nevertheless within $0.004$ of the best we know.

\begin{claim}\label{claim:smallk0}
The triple
\begin{align*} 
    (\upsilon,\alpha,\beta) 
    &= \left(3,\frac{193779720907921494583047135253}{3\cdot 2^{95}},\frac{12202340518338097915331656565}{9\cdot 2^{91}}\right) \\
    &\approx (3,\ 1.6305627,\ 0.5476096)
\end{align*}
is a strong WT-triple, witnessed by a random witness with $K=5$ atoms.
\end{claim}

Consequently $b \le 1.8898801$ and $k_0 \le 29.9986982 < 30$.

\begin{proof}
The witness $(3,w,t)$ is depicted in Figures~\ref{fig:b30-w} (for $w$) and~\ref{fig:b30-t} (for $t$). The exact values are available in \url{https://github.com/dcartermath/sidon-error/tree/main/explicit}, and the fact that these functions witness the relevant WT-triple is verified by \url{https://github.com/dcartermath/sidon-error/blob/main/verify.py}. The most interesting property, that $\EV{w\ast t}$ is at least 1 on the interval $[0,\sigma]$, is shown visually in Figure~\ref{fig:b30-Ewstart pics}.
\end{proof}

\begin{figure}[htbp!]
    \centering
    \includegraphics{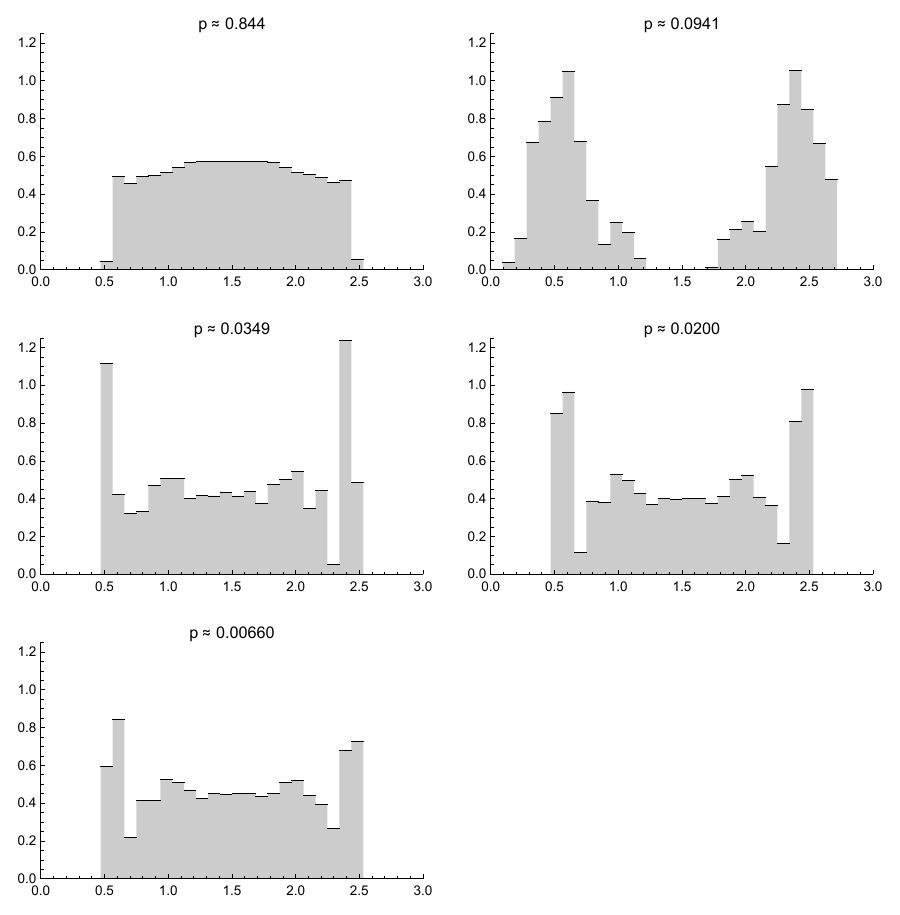}
    \caption{The window for the five atoms of a witness of $(3,1.6305627,0.5476096)$.}
    \label{fig:b30-w}
\end{figure}

\begin{figure}[htbp!]
    \centering
    \includegraphics{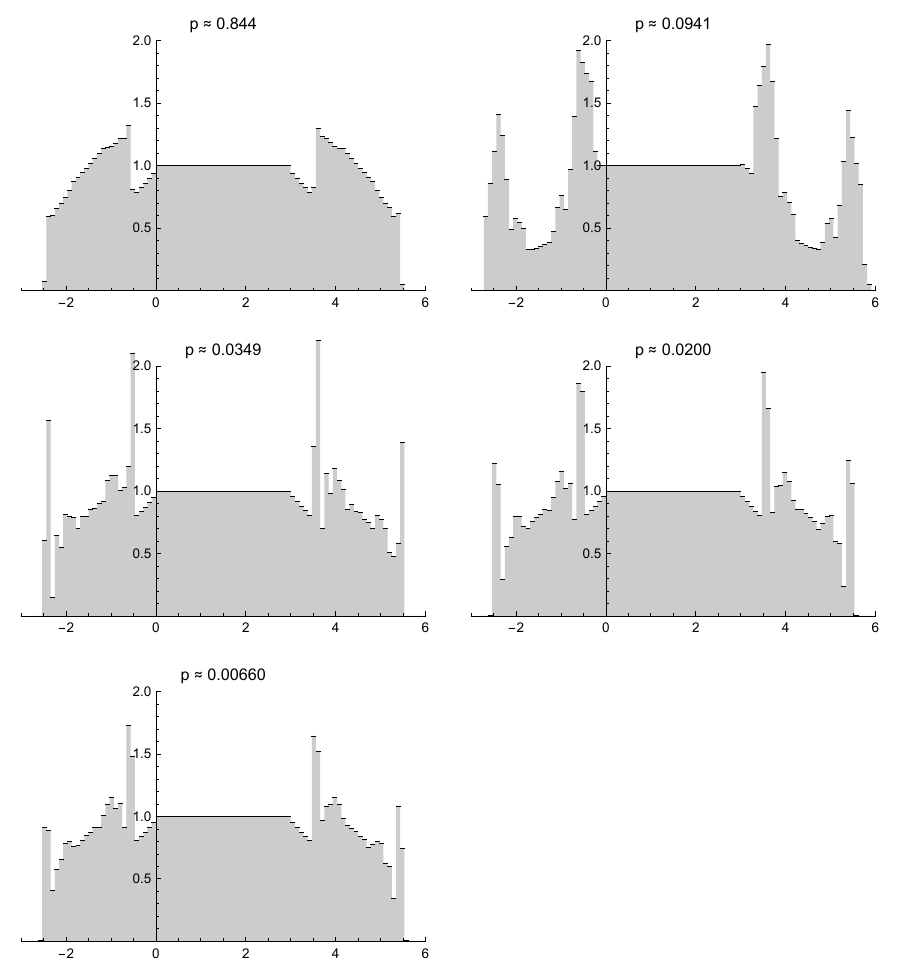}
        \caption{The table for the five atoms of a witness of $(3,1.6305627,0.5476096)$.}
    \label{fig:b30-t}
\end{figure}

\begin{figure}[htbp!]
    \centering
    \includegraphics{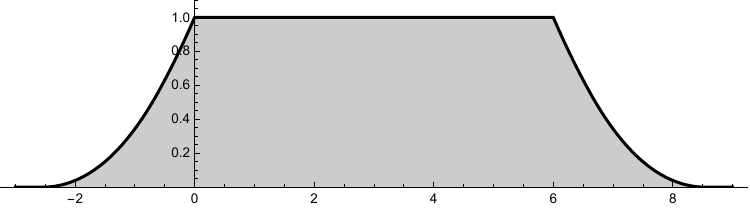}
    \caption{$\EV{w\ast t}$ for the witness to $(3,1.6305627,0.5476096)$.}
    \label{fig:b30-Ewstart pics}
\end{figure}

\subsection{The WT-triple with smallest \texorpdfstring{$b$}{b}}
\label{sec:best triple}
Allowing $k_0$ to grow, we can do better in $b$. The best triple we have found is the
following; note that the gain over Claim~\ref{claim:smallk0} in $b$ is small, and is
paid for with a much larger $k_0$.

\begin{claim}\label{claim:bestb}
The triple
\[ (\upsilon,\alpha,\beta) = (10,\alpha_0,\beta_0 )
\approx (10,\ 0.944814,\ 0.941259) \]
is a WT-triple, where
\[ \alpha_0 = {\scriptstyle\frac{\begin{array}{r}24534927811575735074232072732714727979451900637084863352122\\
68095831410090107899954872437715855867237470270581061955550\\
70756187862800085525830942972173964247560024026270654818106\\
70864505766255992189006831208767880323636547985259235591445\\
13867197246825055699735993342348910137052272035579462663787
\end{array}}{\begin{array}{r}25968024015819018585255749413962888050874260032839104696162\\
07527338283601181039792274031559691376644227380078582776222\\
47583366127645288213594143153013277322919052741726652238862\\
15485406618227891413397956115521280961039716282295125506672\\
24642669840915160749996375246335775289250052685000000000000
\end{array}}} \]
and
\[ \beta_0 = \frac{1772130384572488116897712811968136530510904755768535954960651}{1882725390027397260273972602940280969126497650532485399838720}. \]
\end{claim}

Consequently $b \le 1.88607$.

\begin{proof}
The witness $(10,w,t)$ is depicted in Figures~\ref{fig:b-w} (for $w$) and~\ref{fig:b-t} (for $t$). The exact values are available in \url{https://github.com/dcartermath/sidon-error/tree/main/asymptotic}, and the fact that these functions witness the relevant WT-triple is verified by \url{https://github.com/dcartermath/sidon-error/blob/main/verify.py}. The most interesting property, that $\EV{w\ast t}$ is at least 1 on the interval $[0,\sigma]$, is shown visually in Figure~\ref{fig:b-Ewstart pics}.
\end{proof}

\begin{figure}[htbp!]
    \centering
    \includegraphics{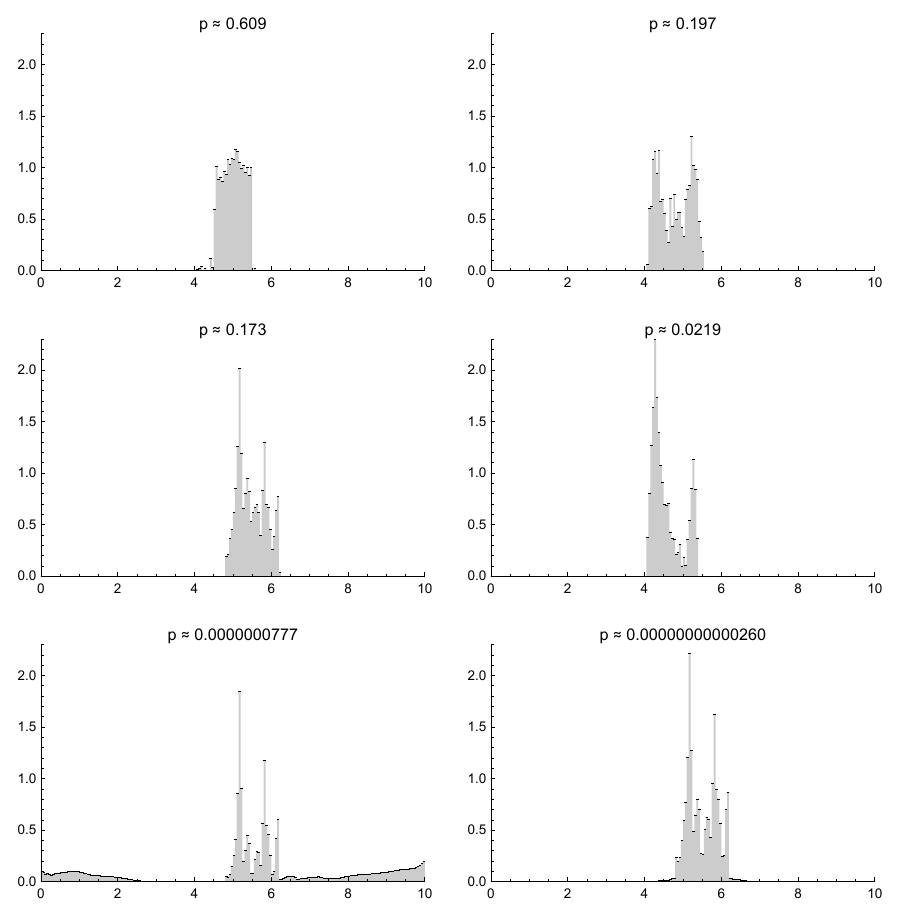}
    \caption{The window for the six atoms of a witness of $(10,0.944814,0.941259)$.}
    \label{fig:b-w}
\end{figure}

\begin{figure}[htbp!]
    \centering
    \includegraphics{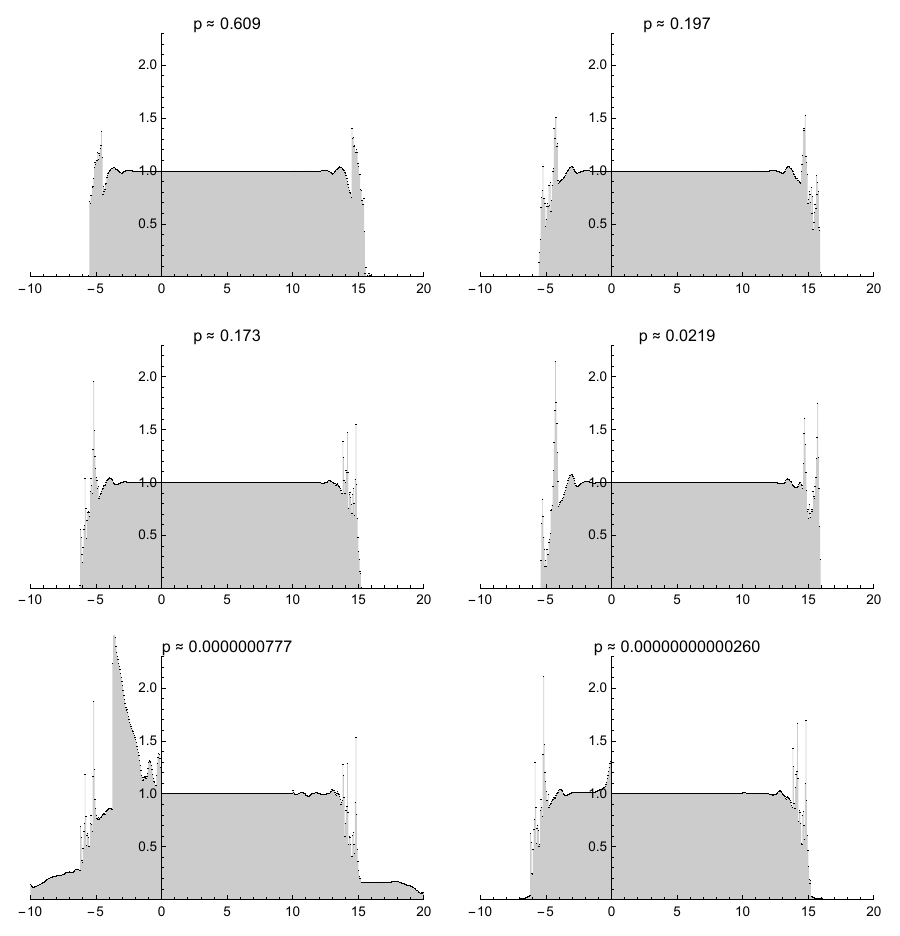}
        \caption{The table for the six atoms of a witness of $(10,0.944814,0.941259)$.}
    \label{fig:b-t}
\end{figure}

\begin{figure}[htbp!]
    \centering
    \includegraphics{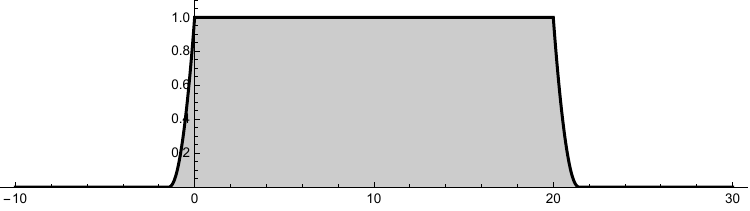}
    \caption{$\EV{w\ast t}$ for the witness to $(10,0.944814,0.941259)$.}
    \label{fig:b-Ewstart pics}
\end{figure}

\end{document}